\documentclass[a4paper,11pt]{article}
\usepackage{stmaryrd}
\usepackage{graphicx}
\usepackage{amsfonts}
\usepackage{amsmath,amsthm}
\usepackage{amssymb}
\usepackage{mathrsfs}
\usepackage{indentfirst}
\usepackage{titlesec}
\usepackage{caption2}
\usepackage{color}
\usepackage[normalem]{ulem}
\usepackage{algorithm}
\usepackage{algorithmic}
\usepackage{lineno}

\usepackage[english]{babel}
\usepackage{cite}
\usepackage{bbm}
\def\rd{{\rm d}}

\usepackage[hypertexnames=false]{hyperref}

\newtheorem{theorem}{Theorem}
\newtheorem{definition}{Definition}

\newtheorem{assumption}{Assumption}
\newtheorem{proposition}{Proposition}
\newtheorem{remark}{Remark}
\newtheorem{lemma}{Lemma}
\newtheorem{corollary}{Corollary}
\newtheorem{example}{Example}

\usepackage{etoolbox}
\makeatletter
\patchcmd{\ttlh@hang}{\parindent\z@}{\parindent\z@\leavevmode}{}{}
\patchcmd{\ttlh@hang}{\noindent}{}{}{}
\makeatother

\begin{document}

\title{On the Posterior Distribution of a Random Process Conditioned on Empirical Frequencies of a Finite Path:\\
the i.i.d and finite Markov chain case}

\author{
Wenqing Hu\thanks{Department of Mathematics and Statistics, Missouri University of Science and Technology
(formerly University of Missouri, Rolla). Email: \texttt{huwen@mst.edu}
}
\   and   
Hong Qian\thanks{Department of Applied Mathematics, University of Washington. Email: \texttt{hqian@uw.edu}}
}

\date{}

\maketitle

\begin{abstract}
We obtain the posterior distribution of a random process conditioned on observing the empirical frequencies of a finite sample path. We find under a rather broad assumption on the ``dependence structure'' of the process, {\em c.f.} independence or Markovian, the posterior marginal distribution of the process at a given time index can be identified as certain empirical distribution computed from the observed empirical frequencies of the sample path. We
show that in both cases of discrete-valued i.i.d. sequence and finite Markov chain, a certain ``conditional symmetry" given by the observation of the empirical frequencies leads to the desired result on the posterior distribution. Results for both finite-time observations and its asymptotic infinite-time limit are connected via the idea of Gibbs conditioning. Finally, since our results demonstrate a central role of the empirical frequency in understanding the information content of data, we use the Large Deviations Principle (LDP) to construct a general notion of ``data-driven entropy", from which one can apply a formalism from the recent study of statistical thermodynamics to data.
\end{abstract}

\textit{Keywords}: Posterior Distribution, Empirical Frequency, Gibbs Conditioning, Large Deviations Principle, Statistical Thermodynamics.

\textit{2020 Mathematics Subject Classification Numbers}: 60J10, 60F10, 62C10, 62C12.

\section{Introduction}\label{Sec:Intro}

Ever since the axiomatic construction of A. N. Kolmogorov \cite{KolmogoroffFoundations}, Probability Theory always starts with a probability space $(\Omega, \mathcal{F}, \mathbf{P})$, where $\Omega$ is a sample space of elementary events, $\mathcal{F}$ is a $\sigma$-algebra on the sample space, and $\mathbf{P}$ is a probability measure given {\it a priori}.  A major concern when applying this paradigm to real-world problems, identified as the ``second-half of probability theory'' by E. T. Jaynes \cite{JaynesBook}, the pioneer of Maximum Entropy Principle, lies in that one usually does not know the probability measure in any realistic way.  As a matter of fact the $\mathbf{P}$ is not a scientific observable which should be a function on $\Omega$. This has not prevented some researchers relying on the ``frequentist" point of view and use observed empirical distribution as a substitution for the real probability. However, when one ``fits" data to a statistical model, the form of the unknown probability distribution necessarily enters as assumptions, hidden or explicitly.  This is most succinctly pointed out by P. J. Huber in \cite{HuberBook}:
\begin{quote}
{\em Statistical inferences are based only in part upon the observations. An equally important base is formed by prior assumptions about the underlying situation.  Even in the simplest cases, there are explicit or implicit assumptions about randomness and independence, about distributional models, perhaps prior distributions for some unknown parameters, and so on.
}
\end{quote}
The route of building statistical models from probability, as an integral part of data science, thus, should always start with the ``basic assumption of a probability space including a prior probability measure".  This realization led to the rise of Bayesian logic in which posterior conditional probability becomes a central object, while in many science and engineering, empirical frequency is still a legitimate surrogate for the probability. Inspired by the upcoming ``big data'' in the near future, the main concern of the present work is to secure a probability measure from empirical frequency data.  We are particularly motivated by the following question:

\begin{quote}
\noindent\textit{Given an observation of the empirical frequencies of a random process, to what extent can we recover the probability structure of the original random process via conditioning?} 
\end{quote}

Our main rationale for addressing the question lies as follows: Under a rather broad assumption of the dependence structure on the process, such as ``independence'' or ``Markovian'', we can identify the posterior marginal distribution of the process at a given time index as certain empirical distribution that is computable from the empirical frequencies of observed outcomes of the process, {\em e.g.} a finite sample path. In the simplest case for an i.i.d sequence, the distribution from the computation is the empirical frequencies of the outcomes. In general one considers ``sample path frequencies" admitted by the observed empirical frequencies of the outcomes, through an analysis of sample path trajectories. A summarizing conclusion in short: ``posterior probability distribution is given by some kind of empirical frequency based on observations". There is a quite related idea under the notion of de Finetti's {\em exchangeability} \cite{Diaconisbook}.  In particular Diaconis and Freedman \cite{diaconis-freedman} have constructed a symmetry based definition for the very notion of Markov dependence in a sequence.  Our work, however, studies the posterior conditional probability within the framework of a Markov process by exploring such a symmetry. Indeed, the results we demonstrate below can be easily extended to the exchangeble case since all we need is the symmetry provided by exchangebility.

In this work, we carry out the analysis for the two cases: (1) a sequence of i.i.d random variables taking discrete values, and (2) a sequence of random variables that follow a finite-state Markov chain. The former is included since some of the subtleties for the latter are more clearly exhibited.  In both cases, the fact that the posterior marginal distribution is given by the empirical distribution can be roughly understood as a result of ``conditional symmetry" for the joint distribution of the process under interest. To illustrate this simply:  Suppose we have observed the empirical frequencies of a sequence of $n$ i.i.d. random variables $X_1,....,X_n$ taking positive integer values. This implies that the number of $X$'s taking value $i\in \mathbb{N}$ is a known number of count $\nu_i\in \mathbb{Z}_+$, such that $\sum_{i\in \mathbb{N}}\nu_i=n$. When conditioned on the observation of only the quantities $\nu_i, i\in \mathbb{N}$, the random variable $X_\ell$ at a fixed index $1\leq \ell\leq n$ may take the value of any number in the unordered list (it may contain repeated elements) $[i_1,...,i_n]$ such that $\text{\#}\{\ell: i_\ell=i\}=\nu_i$ for any $i\in \mathbb{N}$. However, using the i.i.d. (exchangeable) property, it is intuitively true that $X_\ell$ may take each of the possible $i_\ell$'s with the same conditioned probability. This is a result of the fact that any sample trajectory $X_1,...,X_n$ of an i.i.d. sequence conditioned on the observed frequencies has the same joint distribution. Such a conditional symmetry leads to the fact that the observed empirical frequencies give the posterior marginal of $X_\ell$. 

While simple to understand, the situation is more involved when the i.i.d. sequence is replaced by a finite-state Markov chain. In this case, upon the observation of $n+1$ steps $Y_1,Y_2,...,Y_n,Y_{n+1}$ in the chain, we count the empirical frequencies of the occurrences of one-step transitions $X_1=(Y_1,Y_2),...,X_n=(Y_n,Y_{n+1})$. Given the frequencies of these transitions and fix the starting point of the Markov chain $Y$, we observe that the joint probability distribution of $X_1,...,X_n$ remains the same regardless of how we do admissible permutations of the occurrences of the one-step transitions. Here an admissible permutation means that the resulting sequence of the $X$'s will still form a trajectory of a Markov chain. Thus we obtained ``conditional symmetry" at the level of sample path trajectories. As a result, the posterior marginal distribution of $X_1$ upon observing the empirical frequencies of $X_1,...,X_n$, is given by the empirical distribution of admissible trajectories that match the observed empirical frequencies. This means that the posterior marginal distribution of $X_1$ is proportional to the number of all admissible trajectories of $X_1, ..., X_n$ that match the observed empirical frequencies. In this case, we can still understand the posterior distribution as some kind of ``empirical frequency", but it is at the level of sample path trajectories rather than the above direct observation of the $X$'s. 

Problems of similar nature, namely the calculation of posterior distributions of a random process conditioned on given observations, have been considered in previous works, see \cite{touchette_2015} and the references cited within. The key difference is that we consider finite-time sample path while \cite{touchette_2015} considered conditioning under large deviation principle (LDP), {\it e.g.} with infinite-time limit. To connect these two types of results, we investigate the infinite-time limit 
of our finite-time results.  By using combinatorial enumeration results known as Whittle's formula \cite{whittle1955, billingsly1961AnnStat}, we are able to show that for the finite Markov chain under additional ergodicity condition, in the asymptotic limit as the number of observations tends to infinity, the numbers of admissible trajectories starting from different one-step transitions are evenly distributed. This yields the fact that in the infinite-time limit, the ``conditional symmetry" at the level of sample paths for ergodic finite Markov chains is reduced to the ``conditional symmetry" at the level of the observations of one-step transitions. Therefore for an ergodic finite Markov chain, as the number of observations tends to infinity, the posterior distribution of the one-step transition tends to be given asymptotically by the empirical frequency of the observed one-step transitions, a result that aligns nicely with the i.i.d. case.  Notice that under the ergodicity assumption, the latter empirical frequency also converges to the actual transition probabilities multiplied by the stationary measure. In fact, such an asymptotic result can also be obtained from the general principle of \textit{Gibbs conditioning} (see Sec. 3.3 of \cite{DemboZeitouniLDPBook} and Sec. 7.3 of \cite{IEEEConditioning}). However, our approach is more systematic and reveals more clearly how different levels of symmetry at play in obtaining the limit of the posterior distribution, and it is expected to be generalizable to broader classes of random processes. 

Our result can be interpreted more conceptually from a measure-theoretic point of view since the conditional symmetry that we revealed here is simply a result of the product structure of the underlying probability measure defining the process. We expect that a general principle should hold beyond the cases we can rigorously demonstrate here, that a ``conditional symmetry" leads to the procedure of using empirical distributions calculated from observed frequencies to stand for posterior marginal distributions. This principle is expected to be also applicable in understanding the posterior marginal distribution of continuous-time Markov processes, for which we leave the exact formulation of the result and its proof to a future time. 

As was pointed out in \cite{EAMSNotices}, there have been two different paradigms for doing scientific research: the Keplerian paradigm, which is the data-driven approach, and the Newtonian paradigm, which is considered the first-principle-based approach. The natural question that arises with the rapid advancement of current data science is to understand how much information one can extract from large collection of data. Essentially, the largest amount of information we can extract from data is the maximum capability of doing science under the Keplerian paradigm. Our work thus can be regarded as providing an understanding to this question at least from one perspective. The underlying fact is: Given only the frequency observations of a random process, a conditional symmetry occurs to the posterior distribution, and no more additional information can be extracted from frequency observations. This illustrates the importance of empirical frequency observations when working under the Keplerian paradigm. Based on this, one can further apply the Large Deviations Principle \cite{DemboZeitouniLDPBook,FWbook} to the empirical frequencies either at the level of random variable or at the level of sample paths. Following this approach, as we shall illustrate, one can achieve a natural integration of the {\em frequentist school} and the {\em Bayesian school}, with the concept of {\em entropy} emerging at the center stage. The entropy concept here, however, is much more general and broader than that of Gibbs, which appears in statistical thermodynamics, and Shannon's which defines current information theory. It can be regarded as an abstract ``data-driven" entropy from which we can apply the formalism of thermodynamics to data sciences. For asymptotic infinitely large samples and in the context of LDP, this is the essential idea of Gibbs conditioning. 

The paper is organized as follows: In Section \ref{Sec:IID} we derive the results when the process $X_1,...,X_n$ consists of an i.i.d. sequence; in Section \ref{Sec:MarkovChain} we consider the finite Markov chain case; in Section \ref{Sec:InfiniteTimeLimitMarkovChain} we study the infinite-time limit of the results in Sec. \ref{Sec:MarkovChain}; in Section \ref{Sec:discussions} we discuss a novel understanding of the mathematical content of statistical thermodynamics, a banch of theoretical physics, in terms of the large deviations theory and Koopman-Damois exponential family of models. This gives the paradigm of Maximum Entropy Principle and a deep connection to an overlooked theory developed by L. Szilard \cite{szilard} and B. B. Mandelbrot \cite{mandelbrot_1962}.  We further propose some conjectures and provide an outlook for future research.

\section{Conditioned on sample frequency: The i.i.d case}\label{Sec:IID}

Let $X_1$, ..., $X_n$, ... be an i.i.d sequence defined on the probability space $(\Omega, \mathcal{F}, \mathbf{P})$ with common distribution as a random variable $X$ taking values in $\mathbb{N}$. Given a sequence of sample frequencies $\nu_k\in \mathbb{N}_+$ satisfying 
$$\sum_{k\in\mathbb{N}} \nu_k = n \ , $$
we consider the event
\begin{equation}\label{Eq:IIDSampleFrequencyEvent}
\mathcal{E}_{\{\nu_k\} } = \left\{ \sum_{\ell=1}^n \mathbbm{1}_k(X_\ell) = \nu_k, k\in\mathbb{N} \right\},
\end{equation}
where $\mathbbm{1}_k(X_\ell)=\left\{\begin{array}{ll}1 \ , & \text{ if } X_{\ell}=k \ ,
\\
0 \ , & \text{ otherwise .}
\end{array}\right.$ Thus $\mathcal{E}_{\{\nu_k\}}$ stands for the event that the trajectory $X_\ell, \ell=1,2,...,n$ takes on value $k$ with frequency $\nu_k$, $k\in \mathbb{N}$, respectively. The event $\mathcal{E}_{\{\nu_k\}}$ can be viewed as the ``observation" of the trajectory up to time $\ell=n$. Conditioned on this event, we claim the following

\begin{theorem}[posterior distribution for the i.i.d. case]\label{Thm:IIDPosteriorEqualsPrior}
Given $m\in \mathbb{N}$ and any $1\leq \ell \leq n$, we have
\begin{equation}\label{Thm:IIDPosteriorEqualsPrior:Eq}
\mathbf{P}\big(X_{\ell}=m \big|\mathcal{E}_{\{\nu_k\}}\big) = \frac{\nu_m}{n} \ .
\end{equation}
\end{theorem}

\begin{proof}
Given $\mathcal{E}_{\{\nu_k\}}$, for any $k\in \mathbb{N}$, we know that among each $X_\ell$ in the sequence $X_1,...,X_n$, $\ell=1,...,n$, the number $k$ occurs at a multiple of $\nu_k$ times. Since each $\nu_k\geq 0$ is a non-negative integer, and $\sum\limits_{k\in \mathbb{N}}\nu_k=n$, we see that there are only finitely many $k$'s such that $\nu_k\geq 1$, and for all the rest of $k\in \mathbb{N}$ we have $\nu_k=0$. We order the $k$'s such that $\nu_k\geq 1$ in an increasing sequence as $1\leq k_1<k_2<...<k_I<\infty$ for some finite $I\in \mathbb{N}$. Consider the set of $n$ \textit{distinguished} elements
\begin{equation}\label{Thm:IIDPosteriorEqualsPrior:Eq:DistinguishedOccurances}
\mathcal{X}^{\text{distinguished}}\equiv \bigcup_{i=1}^I\{(k_i,1),...,(k_i,\nu_{k_i})\} \ .
\end{equation} 
Let us randomly pick each element in $\mathcal{X}^{\text{distinguished}}$ once and without replacement, so that we can establish $n!$ sequences of length $n$ with each sequence $\mathscr{S}$ consists of ordered elements of the form $(k,a)$ for some $k\in \{k_1,...,k_I\}$ and some $1\leq a\leq \nu_k$. An element-wise projection $\mathfrak{Y}$ with $(k,a)\stackrel{\mathfrak{Y}}{\rightarrow} k$ applied to each of the above sequence results in an outcome of the sequence $X_1,...,X_n$ conditioned on $\mathcal{E}_{\{\nu_k\}}$.  We construct a probability measure $\mathscr{P}$ on the space of all such sequences $\mathscr{S}$, such that for each sequence we have
\begin{equation}\label{Thm:IIDPosteriorEqualsPrior:Eq:ProbabilityMeasureOnDistinguishedSequences:Symmetry}
\mathscr{P}(\mathscr{S})=C
\prod_{i=1}^I \left(\mathbf{P}(X=k_i)\right)^{\nu_{k_i}}
\end{equation} 
for an undetermined normalizing constant $C>0$. For a given outcome $(i_1,...,i_n)$ of the sequence $X_1,...,X_n$, such that when counting frequencies, the sequence $(i_1,...,i_n)$ satisfies $\mathcal{E}_{\{\nu_k\}}$, we must have 
$\{X_1=i_1,...,X_n=i_n\}\subset \mathcal{E}_{\{\nu_k\}}$. Therefore for such sequences we have $\mathbf{P}(X_1=i_1,...,X_n=i_n, \mathcal{E}_{\{\nu_k\}})=\mathbf{P}(X_1=i_1,...,X_n=i_n)=\displaystyle{\prod_{i=1}^I \left(\mathbf{P}(X=k_i)\right)^{\nu_{k_i}}}$. Thus
we have 
\begin{equation}\label{Thm:IIDPosteriorEqualsPrior:Eq:ProbabilityMeasureOnOutcomeSequenceGivenFrequencies}
\mathbf{P}(X_1=i_1,...,X_n=i_n|\mathcal{E}_{\{\nu_k\}})=D\displaystyle{\prod_{i=1}^I \left(\mathbf{P}(X=k_i)\right)^{\nu_{k_i}}} \ ,
\end{equation}
where $D>0$ is an undetermined normalizing constant for the conditional probability measure $\mathbf{P}(\bullet|\mathcal{E}_{\{\nu_k\}})$. Comparing (\ref{Thm:IIDPosteriorEqualsPrior:Eq:ProbabilityMeasureOnDistinguishedSequences:Symmetry}) and (\ref{Thm:IIDPosteriorEqualsPrior:Eq:ProbabilityMeasureOnOutcomeSequenceGivenFrequencies}), we see that for any given 
 outcome $(i_1,...,i_n)$ of the sequence $X_1,...,X_n$ satisfying $\mathcal{E}_{\{\nu_k\}}$ we must have

\begin{equation}\label{Thm:IIDPosteriorEqualsPrior:Eq:ProjectionProbabilityIdentity}
\mathbf{P}(X_1=i_1,...,X_n=i_n|\mathcal{E}_{\{\nu_k\}})=K\mathscr{P}(\mathscr{S}) \ ,
\end{equation}
for some constant $K>0$. 

Combining (\ref{Thm:IIDPosteriorEqualsPrior:Eq:ProbabilityMeasureOnDistinguishedSequences:Symmetry}), (\ref{Thm:IIDPosteriorEqualsPrior:Eq:ProbabilityMeasureOnOutcomeSequenceGivenFrequencies}) and (\ref{Thm:IIDPosteriorEqualsPrior:Eq:ProjectionProbabilityIdentity}), we see the following two facts

\begin{itemize}
\item[Fact 1:] For each of the $n!$ different sequences $\mathscr{S}$, $\mathscr{P}(\mathscr{S})$ has the same value.
\item[Fact 2:] The sequence $\mathscr{S}$ in the RHS of (\ref{Thm:IIDPosteriorEqualsPrior:Eq:ProjectionProbabilityIdentity}) can be an arbitrary length-$n$ sequence picked from the $n!$ possible choices.
\end{itemize}

For each sequence $(i_1,...,i_n)$, we collect all possible sequences $\mathscr{S}$ such that $\mathfrak{Y}(\mathscr{S})=(i_1,...,i_n)$. We claim that we have 

\begin{equation}\label{Thm:IIDPosteriorEqualsPrior:Eq:ProjectionProbabilityIdentityUnionS}
\mathbf{P}(X_1=i_1,...,X_n=i_n|\mathcal{E}_{\{\nu_k\}})=\mathscr{P}\left(\text{all possible }\mathscr{S} \text{ such that } \mathfrak{Y}(\mathscr{S})=(i_1,...,i_n)\right) \ .
\end{equation}

This is because due to the above Fact 1, we have
$$\begin{array}{ll}
& \mathscr{P}\left(\text{all possible }\mathscr{S} \text{ such that } \mathfrak{Y}(\mathscr{S})=(i_1,...,i_n)\right)
\\
= & \left(\text{Number of all possible } \mathscr{S} \text{ such that } \mathfrak{Y}(\mathscr{S})=(i_1,...,i_n)\right)\cdot \mathscr{P}(\mathscr{S}) \ .
\end{array}$$

Now we note that by simple combinatorics we have

\begin{itemize}
\item[Fact 3:] For each realization $(i_1,...,i_n)$ of $X_1,...,X_n$,
$$\left(\text{Number of all possible } \mathscr{S} \text{ such that } \mathfrak{Y}(\mathscr{S})=(i_1,...,i_n)\right)$$
is independent of the choice of $(i_1,...,i_n)$\footnote{This number is actually $\nu_{k_1}\!! \, \nu_{k_2}\!! \, ... \, \nu_{k_I}\!!$ \ .}. 
\end{itemize}

Therefore by Fact 3 we see that (\ref{Thm:IIDPosteriorEqualsPrior:Eq:ProjectionProbabilityIdentityUnionS}) is equivalent to 

$$\mathbf{P}(X_1=i_1,...,X_n=i_n|\mathcal{E}_{\{\nu_k\}})=K_1\mathscr{P}(\mathscr{S}) \ .$$
It is then easy to see that $K_1=K$ just by normalization of the probability measure $\mathscr{P}(\bullet)$ and the conditional probability measure $\mathbf{P}(\bullet|\mathcal{E}_{\{\nu_k\}})$, as well as the above Fact 2. So we proved that (\ref{Thm:IIDPosteriorEqualsPrior:Eq:ProjectionProbabilityIdentityUnionS}) is valid.

From (\ref{Thm:IIDPosteriorEqualsPrior:Eq:ProjectionProbabilityIdentityUnionS}) we know that for an element $m\in \{k_1,...,k_I\}$ we have

\begin{equation}\label{Thm:IIDPosteriorEqualsPrior:Eq:PosterialMarginalDecomposition:Step:Counting}
\mathbf{P}(X_\ell=m|\mathcal{E}_{\{\nu_k\}})
= \sum\limits_{a=1}^{\nu_m} \mathscr{P}\left(\text{all possible }\mathscr{S} \text{ whose } \ell\text{-th element is } (m,a)\right) \ .
\end{equation}

We claim that for each element $(m,a)\in \mathcal{X}^{\text{distinguished}}$ we also have

\begin{itemize}
\item[Fact 4:] For each $(m,a)\in \mathcal{X}^{\text{distinguished}}$, 
$$\left(\text{Number of all possible } \mathscr{S} \text{ whose } \ell \text{-th element is } (m,a)\right)$$
is independent of $(m,a)$ \footnote{This number is actually $(n-1)!$ \ .}. 
\item[Fact 5:] $\mathscr{P}\left(\text{all possible }\mathscr{S} \text{ whose } \ell\text{-th element is } (m,a)\right)\stackrel{\text{def}}{=}p$ is independent of $(m,a)$.
\end{itemize}

The above Fact 4 is a simple combinatorial observation, and Fact 5 is a consequence of Facts 1 and 4. Since $\mathcal{X}^{\text{distinguished}}$ contains $n$ elements, by Fact 5 we know that $np=1$, i.e., $p=\dfrac{1}{n}$. This fact combined with (\ref{Thm:IIDPosteriorEqualsPrior:Eq:PosterialMarginalDecomposition:Step:Counting}) give us (\ref{Thm:IIDPosteriorEqualsPrior:Eq}). Note that when $m\not\in \{k_1,...,k_I\}$ we have $\nu_m=0$ and (\ref{Thm:IIDPosteriorEqualsPrior:Eq}) is trivial. So we have proved the whole statement. 
\end{proof}

\begin{remark}[Alternate proof of Theorem \ref{Thm:IIDPosteriorEqualsPrior}]\rm\label{Remark:AlternateProof:Prop:IIDPosteriorEqualsPrior}
One can establish a less intuitive but more direct combinatorial proof of Theorem \ref{Thm:IIDPosteriorEqualsPrior} as follows.
\begin{proof}[Alternate Proof of Theorem \ref{Thm:IIDPosteriorEqualsPrior}]
Given $\mathcal{E}_{\{\nu_k\}}$, the occurrences of $X_{\ell}=m$ happen on an arbitrary length-$\nu_m$ sub-index sequence $1\leq \ell_1< ... <\ell_{\nu_m}\leq n$ such that $\ell=\ell_1, ...., \ell_{\nu_m}$. 

Consider the event
$$\mathcal{E}^m_{(\ell_1,...,\ell_{\nu_m})}\equiv \{X_{\ell_1}=...=X_{\ell_{\nu_m}}=m \text{ for } 1\leq \ell_1< ... <\ell_{\nu_m}\leq n \text{ and } X_\ell\neq m \text{ for } \ell\neq \ell_1,...,\ell_m\} \ .$$
Then for two distinct sequences $(\ell_1,...,\ell_{\nu_m})\neq (\widetilde{\ell}_1,...,\widetilde{\ell}_{\nu_m})$ we must have $\mathcal{E}^m_{(\ell_1,...,\ell_{\nu_m})}\cap \mathcal{E}^m_{(\widetilde{\ell}_1,...,\widetilde{\ell}_{\nu_m})}=\emptyset$.

Let us also define the event 
$$\check{\mathcal{E}}_{\{\nu_k\}}^{m,\nu_m}=\left\{\sum_{\ell=1}^{n-\nu_m} \mathbbm{1}_k(X_\ell) = \nu_k, k\in\mathbb{N}\backslash\{m\}\right\} \ .$$

Define $\mathbf{P}\left(\mathcal{E}^m_{(\ell_1,...,\ell_{\nu_m})}\right)=p$, then it is easy from the above that 
$$\mathbf{P}\left(\mathcal{E}_{\{\nu_k\}}\right)= {n\choose \nu_m}p \cdot \mathbf{P}\left(\check{\mathcal{E}}_{\{\nu_k\}}^{m,\nu_m}\right) \ ,$$

$$\mathbf{P}\left(X_{\ell}=m \ , \ \mathcal{E}_{\{\nu_k\}}\right)= {n-1\choose \nu_m-1}p \cdot \mathbf{P}\left(\check{\mathcal{E}}_{\{\nu_k\}}^{m,\nu_m}\right) \ .$$

Thus this implies that $$\mathbf{P}\left(X_\ell=m|\mathcal{E}_{\{\nu_k\}}\right)=\dfrac{{n-1\choose \nu_m-1}}{{n \choose \nu_m}}=\dfrac{\nu_m}{n} \ , $$ which is \eqref{Thm:IIDPosteriorEqualsPrior:Eq}.
\end{proof}
\end{remark}

\begin{remark}[Conditional Symmetry]\rm\label{Remark:IIDProofIllustratesConditionalSymmetry}
In the above proof, we have extended the conditional probability $\mathbf{P}(\bullet|\mathcal{E}_{\{\nu_k\}})$ on the space of all outcome sequences of $X_1,...,X_n$ that match the frequency outcome $\mathcal{E}_{\{\nu_k\}}$, to the probability $\mathscr{P}(\bullet)$ on an ``lifted" probability space of sequences $\mathscr{S}$. In this correspondence, the Facts 1,2,3 lead to (\ref{Thm:IIDPosteriorEqualsPrior:Eq:PosterialMarginalDecomposition:Step:Counting}), which states that the conditional probability under our interest is equal to an absolute probability given by $\mathscr{P}$. This further helps us to understand that the ``conditional symmetry" is represented by an ``absolute symmetry", exactly stated as the Fact 1 in the above proof.

Actually, even without lifting the conditional probability measure $\mathbf{P}(\bullet|\mathcal{E}_{\{\nu_k\}})$ to the absolute probability measure $\mathscr{P}(\bullet)$, the conditional symmetry can still be easily seen from (\ref{Thm:IIDPosteriorEqualsPrior:Eq:ProbabilityMeasureOnOutcomeSequenceGivenFrequencies}), which is a result of i.i.d property. By (\ref{Thm:IIDPosteriorEqualsPrior:Eq:ProbabilityMeasureOnOutcomeSequenceGivenFrequencies}), we know that the joint probability of $X_1,...,X_n$ conditioned on $\mathcal{E}_{\{\nu_k\}}$ will remain the same regardless of how we place the outcomes $i_1,...,i_n$. 

The reason why we would like to lift the original conditional probability to the absolute probability $\mathscr{P}$ lies behind our stated Facts 4 and 5 in the above proof. Here we would like to explore another type of symmetry: the ``combinatorial symmetry" that comes from direct counting. By Fact 4, such kind of symmetry asserts that the number of certain combinatorial configurations are the same regardless of specific constraints. This enables us to directly obtain the fact that the empirical frequency of observed outcomes is the posterior marginal distribution. We will see that such ``combinatorial symmetry" may be \emph{broken} in the finite Markov chain case.

Our strategy of proof, although a bit more complicated than a direct combinatorial proof (see Remark \ref{Remark:AlternateProof:Prop:IIDPosteriorEqualsPrior}), reveals the more general symmetric structure of the problem.
\end{remark}

\begin{remark}\rm\label{Remark:IIDPosteriorBreakIndependence}
It can be easily seen using the same argument as in the proof of Theorem 
\ref{Thm:IIDPosteriorEqualsPrior} that the joint conditional distribution 
$$\mathbf{P}(X_{\ell_1}=m_1, X_{\ell_2}=m_2|\mathcal{E}_{\{\nu_k\}})=\dfrac{\nu_{m_1}\nu_{m_2}}{n(n-1)} \ .$$
This together with Theorem \ref{Thm:IIDPosteriorEqualsPrior} imply that $X_{\ell_1}$ and $X_{\ell_2}$ are not conditionally independent. In a heuristic explanation, the conditioning puts a ``nonlinear constraint" on the joint distributions which makes them dependent.
\end{remark}

\section{Conditioned on sample frequency: The finite Markov chain case}\label{Sec:MarkovChain}

Let $Y_1, ..., Y_n, ...$ be a time-homogeneous Markov chain with finite state space $\Sigma=\{1,...,N\}$, $|\Sigma|=N$. Let the transition probability matrix of the process $\{Y_\ell\}_{\ell\geq 1}$ be given by $P=(p_{ij})_{1\leq i, j\leq N}$.
Assume the process starts from an initial probability distribution $\pi^0=(\pi^0_1,...,\pi^0_N)$, $0\leq \pi^0_i\leq 1$, $\sum\limits_{i=1}^N \pi^0_i=1$, such that $\mathbf{P}(Y_1=i)=\pi^0_i$. 

Define the ``consecutive pair" process $X_\ell=(Y_\ell, Y_{\ell+1}), \ell \geq 1$. It is easy to see that $\{X_\ell\}_{\ell\geq 1}$ is also a Markov process with transition probability matrix $$P^{(2)}=(p^{(2)}_{(i,j),(k,l)})_{1\leq i,j,k,l\leq N} \ ,$$ and the matrix elements are computed from $p_{ij}$ via the formula (see \cite[Section 3.1.13]{DemboZeitouniLDPBook})
$p^{(2)}_{(i,j),(k,l)}=\mathbbm{1}_{j=k}\cdot p_{kl}$.

Let us first look more carefully at the procedure by which we transfer a trajectory of $\{Y_\ell\}_{\ell\geq 1}$ to $\{X_\ell\}_{\ell\geq 1}$. The process $\{Y_\ell\}_{\ell=1}^{n+1}$ has induced a measure $\mu$ on $\Sigma^{n+1}$, where each length-$(n+1)$ string $i_1,...,i_n,i_{n+1}$ is measured as
\begin{equation}\label{Eq:pi}
\mu(i_1,...,i_n,i_{n+1})=\pi^0_{i_1}\prod\limits_{\ell=1}^{n} p_{i_\ell i_{\ell+1}} \ .
\end{equation}
The above measure $\mu$ induces a new measure $\mu^{(2)}$ on $(\Sigma\times\Sigma)^n$ such that for each length-$n$ string of the pairs $(i_1,i_2),(i_2,i_3),...,(i_n,i_{n+1})$ we have
\begin{equation}\label{Eq:pi2-nonzero}
\mu^{(2)}((i_1,i_2),(i_2,i_3),...,(i_n,i_{n+1}))=\mu(i_1,...,i_n,i_{n+1})=\pi^0_{i_1}\prod\limits_{\ell=1}^{n} p_{i_\ell i_{\ell+1}} \ .
\end{equation}
Moreover, for any length-$n$ string of the pairs $(i_{11},i_{12}),(i_{21},i_{22}),...,(i_{n1},i_{n2})$ such that there exist some $i_{\ell2}\neq i_{(\ell+1)1}$ for some $1\leq \ell \leq n-1$, we have 
\begin{equation}\label{Eq:pi2-zero}
\mu^{(2)}((i_{11},i_{12}),(i_{21},i_{22}),...,(i_{n1},i_{n2}))=0 \ .
\end{equation}
For future presentation, we introduce the following definition.

\begin{definition}[string of chain type]\label{Def:StringOfChainType}
We denote a length-$n$ string of the pairs 
$$((i_{11},i_{12}),(i_{21},i_{22}),...,(i_{n1},i_{n2}))$$ 
such that $i_{12}=i_{21}, ..., i_{(n-1)2}=i_{n1}$ as a ``\emph{string of chain type}". 
For such a string of chain type, we denote by $i_{11}$ its ``\emph{head}".
\end{definition}

Thus (\ref{Eq:pi2-nonzero}) and (\ref{Eq:pi2-zero}) are saying that $\mu^{(2)}$ only charges on strings of chain type.
Let the sample space for the trajectory $X_1,...,X_n$ be given by $(\Omega=(\Sigma\times\Sigma)^n, \mathcal{F}, \mathbf{P})$, then it is easy to see that $\mu^{(2)}$ is the probability measure for the trajectory $X_1,...,X_n$, i.e., we have
\begin{equation}\label{Prop:MarkovChainPosteriorEqualsPrior:Eq:pi2isP}
\mathbf{P}(X_1=(i_{11},i_{12}),X_2=(i_{21},i_{22}),...,X_n=(i_{n1},i_{n2}))=\mu^{(2)}((i_{11},i_{12}),(i_{21},i_{22}),...,(i_{n1},i_{n2})) \ .
\end{equation}

Given a sequence of sample frequencies $\nu_{(i,j)}\in \mathbb{N}_+$ satisfying 

$$\sum\limits_{i=1}^N \sum\limits_{j=1}^N \nu_{(i,j)}=n \ ,$$
we consider the event
\begin{equation}\label{Eq:MarkovChainSampleFrequencyEvent}
\mathcal{E}_{\{\nu_{(i,j)}\} } = \left\{\sum_{\ell=1}^n \mathbbm{1}_{(i,j)}(X_\ell) = \nu_{(i,j)}, 1\leq i, j\leq N \right\},
\end{equation}
where $\mathbbm{1}_{(i,j)}(X_\ell)=\left\{\begin{array}{ll}1 \ , & \text{ if } X_{\ell}=(i,j) \ ,
\\
0 \ , & \text{ otherwise .}
\end{array}\right.$ Thus $\mathcal{E}_{\{\nu_{(i,j)}\}}$ stands for the event that the trajectory $X_\ell \ (\ell=1,...,n)$ takes on value $(i,j)$ with frequency $\nu_{(i,j)}$, $1\leq i,j\leq N$, respectively. The event $\mathcal{E}_{\{\nu_{(i,j)}\}}$ can be viewed as the ``observation" of the trajectory of $X_\ell$ up to time $\ell=n$. 

Let us suppose that we have observed an outcome of the event $\mathcal{E}_{\{\nu_{(i,j)}\}}$. Then we claim that the Markov chain $Y_\ell$ for a fixed index $1\leq \ell \leq n$ \textit{cannot} be an arbitrary element chosen from the state space $\Sigma=\{1,...,N\}$. This can be seen from the following example.

\begin{example}\rm\label{Example:X1NotArbitrary}
Suppose $\{Y_\ell\}_{\ell\geq 1}$ is a stationary Markov chain with a $3$-element state space $\{1,2,3\}$ and stationary measure $\pi=(\frac{1}{3}, \frac{1}{3}, \frac{1}{3})$.
Set $n=2$ and suppose we have observed 
$$\mathcal{E}_{\{\nu_{(i,j)}\}}=\left\{\nu_{(1,2)}=\nu_{(2,3)}=1, \nu_{(i,j)}=0 \text{ for all other pairs of } (i,j)\right\} \ .$$
Then it is easy to see that $\mathbf{P}(Y_1=1|\mathcal{E}_{\{\nu_{(i,j)}\}})=1$ while $\mathbf{P}(Y_1=2|\mathcal{E}_{\{\nu_{(i,j)}\}})=\mathbf{P}(Y_1=3|\mathcal{E}_{\{\nu_{(i,j)}\}})=0$. This indicates that conditioning on the observed frequencies $\mathcal{E}_{\{\nu_{(i,j)}\}}$ may break stationarity and pick specific possible choices of $Y_\ell \ (1\leq \ell \leq n)$ based on the observed frequencies $\mathcal{E}_{\{\nu_{(i,j)}\}}$.  
\end{example}

Due to the above example, we introduce ``conditional admissible states" for $Y_\ell$ as the following

\begin{definition}[conditional admissible states]\label{Def:ConditionalAdmissibleStates}
Given the stationary Markov chain $\{Y_\ell\}_{\ell=1}^{n+1}$ with state space $\Sigma=\{1,2,...,N\}$
and the observed sample frequencies $\mathcal{E}_{\{\nu_{(i,j)}\}}$ defined in (\ref{Eq:MarkovChainSampleFrequencyEvent}),
we define the ``conditional admissible states" $\Sigma^{\checkmark}(\ell|\mathcal{E}_{\{\nu_{(i,j)}\}})$ 
for $Y_\ell \ (1\leq \ell \leq n)$ 
as the set of all possible choices of $1\leq i\leq N$ such that $\mathbf{P}(Y_{\ell}=i|\mathcal{E}_{\{\nu_{(i,j)}\}})>0$. Thus
$$\Sigma^{\checkmark}(\ell|\mathcal{E}_{\{\nu_{(i,j)}\}})=\left\{i: 1\leq i\leq N \ , \ \mathbf{P}(Y_{\ell}=i|\mathcal{E}_{\{\nu_{(i,j)}\}})>0\right\} \ , \ 1\leq \ell \leq n \ .$$
Given a state $i\in \{1,2,...,N\}$ and some $1\leq \ell\leq n$, we further denote the indicator function 
$$\mathbf{1}^{i, \checkmark}_\ell
\equiv
\mathbf{1}_{\Sigma^{\checkmark}(\ell|\mathcal{E}_{\{\nu_{(i,j)}\}})}(i) \ ,$$
which indicates that state $i$ is conditionally admissible at $Y_\ell$ given the observed frequencies $\mathcal{E}_{\{\nu_{(i,j)}\}}$. 
\end{definition}

Given an outcome of the event $\mathcal{E}_{\{\nu_{(i,j)}\}}$, the sequence $X_1,...,X_n$ must take the form of a string of chain type $X_1=(i_1, i_2), ..., X_n=(i_{n-1}, i_n)$, such that when counting frequencies, the elements in the string satisfy $\mathcal{E}_{\{\nu_{(i,j)}\}}$. We introduce the following

\begin{definition}[number of strings of chain type with given term]\label{Def:NumberStringsGivenTerm}
Given an $(i,j)$ such that $\nu_{(i,j)}\geq 1$ on the event $\mathcal{E}_{\{\nu_{(i,j)}\}}$, we define by \emph{$\text{\#}^{(i,j)}_\ell(\mathcal{E}_{\{\nu_{(i,j)}\}})$} to be the 
number of different strings of chain type $X_1=(i_1,i_2),...,X_n=(i_{n-1}, i_n)$ 
with the $\ell$-th element being $X_\ell=(i,j)$, and satisfying $\mathcal{E}_{\{\nu_{(i,j)}\}}$.
\end{definition}

It is easy to provide the following example showing that for $j_1\neq j_2$ we may have $\text{\#}^{(i,j_1)}_1(\mathcal{E}_{\{\nu_{(i,j)}\}})\neq \text{\#}^{(i,j_2)}_1(\mathcal{E}_{\{\nu_{(i,j)}\}})$. So that Definition \ref{Def:NumberStringsGivenTerm} is non-trivial.

\begin{example}\rm\label{Example:NumberStringsNotEqual}
Suppose $\{Y_\ell\}_{\ell\geq 1}$ is a Markov chain with a $3$-element state space $\{1,2,3\}$. Set $n=3$ and suppose we have observed 
$$\mathcal{E}_{\{\nu_{(i,j)}\}}=\left\{\nu_{(1,2)}=\nu_{(2,1)}=\nu_{(1,3)}=1, 
\nu_{(i,j)}=0 \text{ for all other pairs of } (i,j)\right\} \ .$$
Then it is easy to see that $\text{\#}^{(1,2)}_1(\mathcal{E}_{\{\nu_{(i,j)}\}})=1$ and $\text{\#}^{(1,3)}_1(\mathcal{E}_{\{\nu_{(i,j)}\}})=0$.
\end{example}

With the above definitions, we can compute the posterior marginal probability of $X_1$ conditioned upon observed frequencies $\mathcal{E}_{\{\nu_{(i,j)}\}}$ as the following

\begin{proposition}\label{Prop:MarkovChainPosteriorEqualsPrior}
Given $1\leq i,j\leq N$, then we have
\emph{\begin{equation}\label{Prop:MarkovChainPosteriorEqualsPrior:Eq}
\mathbf{P}\big(X_1=(i,j) \big|\mathcal{E}_{\{\nu_{(i,j)}\}}, Y_1=i\big) = 
\mathbf{1}^{i, \checkmark}_1 
\cdot \dfrac{\mathbf{1}^{j, \checkmark}_2 \cdot \text{\#}^{(i,j)}_1(\mathcal{E}_{\{\nu_{(i,j)}\}})}
{\sum_{k_2=1}^N \mathbf{1}^{k_2, \checkmark}_2 \cdot 
\text{\#}^{(i,k_2)}_1(\mathcal{E}_{\{\nu_{(i,j)}\}})}
\ .
\end{equation}}
Here we follow the convention that if the events $\{Y_1=i\}$ and $\mathcal{E}_{\{\nu_{(i,j)}\}}$ are disjoint, then $\mathbf{P}\big(X_1=(i,j) \big|\mathcal{E}_{\{\nu_{(i,j)}\}}, Y_1=i\big)=0$. 
\end{proposition}

\begin{proof}
The proof follows the same scheme proposed in the i.i.d case (Theorem \ref{Thm:IIDPosteriorEqualsPrior}), but with delicate and interesting differences. Given the two events $\mathcal{E}_{\{\nu_{(i,j)}\}}$ and $\{Y_1=i\}$ and suppose these two events are not disjoint, then for any pair $(i,j)\in \Sigma\times \Sigma$, we know that among each $X_\ell$ in the sequence $X_1,...,X_n$, $\ell=1,2,...,n$, the pair $(i,j)$ occurs at a multiple of $\nu_{(i,j)}$ times. Since each $\nu_{(i,j)}$ is a non-negative integer, and $\sum\limits_{i=1}^N \sum\limits_{j=1}^N \nu_{(i,j)}=n$, we see that there are only finitely many pairs of $(i,j)$'s such that $\nu_{(i,j)}\geq 1$, and for all the rest of $(i,j)\in \Sigma\times \Sigma$ we have $\nu_{(i,j)}=0$. For those pairs of $(i,j)$'s such that $\nu_{(i,j)}\geq 1$, we order them in an alphabetical order as $(i_1,j_1)\prec(i_2,j_2)\prec...\prec(i_I,j_I)$, where $(i_1, j_1)\prec (i_2, j_2)$ if and only if $i_1<i_2$ or $i_1=i_2, j_1<j_2$, and $I\in \mathbb{N}$ is finite. Consider the set of $n$ \textit{distinguished} elements

\begin{equation}\label{Prop:MarkovChainPosteriorEqualsPrior:Eq:DistinguishedOccurances}
\mathcal{X}^{\text{distinguished}}\equiv \bigcup_{\iota=1}^I\{((i_\iota, j_\iota),1),...,((i_\iota, j_\iota),\nu_{(i_\iota, j_\iota)})\} \ .
\end{equation} 
Consider all possible length-$n$ ordered sequences $\mathscr{S}$ consisting of distinguished elements of the form $((i,j),a)\in \mathcal{X}^{\text{distinguished}}$ for some $(i,j)\in \{(i_1,j_1),...,(i_I, j_I)\}$ and some $1\leq a\leq \nu_{(i,j)}$, such that the element-wise projection $\mathfrak{Y}$ with $((i,j),a)\stackrel{\mathfrak{Y}}{\rightarrow} (i,j)$ applied to each of the above sequence $\mathscr{S}$ \emph{results in a string of chain type $\mathfrak{Y}(\mathscr{S})$ with head $i$}. 
Such a string of chain type is an outcome of the sequence $X_1,...,X_n$ conditioned on $\mathcal{E}_{\{\nu_{(i,j)}\}}$ and $\{Y_1=i\}$. We construct a probability measure $\mathscr{P}$ on the space of all such sequences $\mathscr{S}$, such that for each sequence we have

\begin{equation}\label{Prop:MarkovChainPosteriorEqualsPrior:Eq:ProbabilityMeasureOnDistinguishedSequences:Symmetry}
\mathscr{P}(\mathscr{S})=C(i) \prod_{\iota =1}^I p_{(i_\iota, j_\iota)}^{\nu_{(i_\iota, j_\iota)}} 
\end{equation}
for an undetermined normalizing constant $C(i)>0$, that may depend on $i$. For a given outcome $(i_1, i_2)$, $(i_2, i_3)$, ... , $(i_{n-1}, i_n)$ of the sequence $X_1,...,X_n$ such that $i_1=i$ and when counting frequencies, the sequence $(i_1, i_2), (i_2, i_3), ... , (i_{n-1}, i_n)$ satisfies $\mathcal{E}_{\{\nu_{(i,j)}\}}$, we must have $\{X_1=(i, i_2), ..., X_n=(i_{n-1}, i_n)\}\subset \mathcal{E}_{\{\nu_{(i,j)}\}}\cap \{Y_1=i\}$. Therefore for such sequences we must have $\mathbf{P}(X_1=(i, i_2), ..., X_n=(i_{n-1}, i_n), \mathcal{E}_{\{\nu_{(i,j)}\}}, Y_1=i)=\mathbf{P}(X_1=(i, i_2), ..., X_n=(i_{n-1}, i_n))=\pi^0_i \displaystyle{\prod_{\iota =1}^I p_{(i_\iota, j_\iota)}^{\nu_{(i_\iota, j_\iota)}}}$. Since $\mathbf{P}(Y_1=i)=\pi^0_i$ is the initial probability distribution, we further have 

\begin{equation}\label{Prop:MarkovChainPosteriorEqualsPrior:Eq:ProbabilityMeasureOnOutcomeSequenceGivenFrequencies}
\begin{array}{ll}
& \mathbf{P}(X_1=(i, i_2), ..., X_n=(i_{n-1}, i_n)| \mathcal{E}_{\{\nu_{(i,j)}\}}, Y_1=i)
\\
= & \dfrac{1}{\mathbf{P}(\mathcal{E}_{\{\nu_{(i,j)}\}}, Y_1=i)}
\mathbf{P}(X_1=(i, i_2), ..., X_n=(i_{n-1}, i_n), \mathcal{E}_{\{\nu_{(i,j)}\}}, Y_1=i)
\\
= & \dfrac{1}{\mathbf{P}(\mathcal{E}_{\{\nu_{(i,j)}\}}|Y_1=i)\mathbf{P}(Y_1=i)}
\mathbf{P}(X_1=(i, i_2), ..., X_n=(i_{n-1}, i_n), \mathcal{E}_{\{\nu_{(i,j)}\}}, Y_1=i)
\\
= & \dfrac{1}{\mathbf{P}(\mathcal{E}_{\{\nu_{(i,j)}\}}|Y_1=i)}\displaystyle{\prod_{\iota =1}^I p_{(i_\iota, j_\iota)}^{\nu_{(i_\iota, j_\iota)}}}
\\
\equiv & D(i)\displaystyle{\prod_{\iota =1}^I p_{(i_\iota, j_\iota)}^{\nu_{(i_\iota, j_\iota)}}} \ ,
\end{array}
\end{equation}
where the constant $D(i)\equiv \dfrac{1}{\mathbf{P}(\mathcal{E}_{\{\nu_{(i,j)}\}}|Y_1=i)}>0$ 
is an undetermined normalizing constant for the conditional measure $\mathbf{P}(\bullet|\mathcal{E}_{\{\nu_{(i,j)}\}}, Y_1=i)$. 
Comparing (\ref{Prop:MarkovChainPosteriorEqualsPrior:Eq:ProbabilityMeasureOnDistinguishedSequences:Symmetry}) and 
(\ref{Prop:MarkovChainPosteriorEqualsPrior:Eq:ProbabilityMeasureOnOutcomeSequenceGivenFrequencies}), we see that for any given outcome $X_1=(i, i_2), ..., X_n=(i_{n-1}, i_n)$ of the sequence $X_1,...,X_n$ satisfying $\mathcal{E}_{\{\nu_{(i,j)}\}}$ and $\{Y_1=i\}$ we must have 

\begin{equation}\label{Prop:MarkovChainPosteriorEqualsPrior:Eq:ProjectionProbabilityIdentity}
\mathbf{P}(X_1=(i, i_2), ..., X_n=(i_{n-1}, i_n)|\mathcal{E}_{\{\nu_{(i,j)}\}}, Y_1=i)=K(i)\mathscr{P}(\mathscr{S}) \ ,
\end{equation}
for some constant $K(i)>0$. 

Combining (\ref{Prop:MarkovChainPosteriorEqualsPrior:Eq:ProbabilityMeasureOnDistinguishedSequences:Symmetry}), (\ref{Prop:MarkovChainPosteriorEqualsPrior:Eq:ProbabilityMeasureOnOutcomeSequenceGivenFrequencies}) and (\ref{Prop:MarkovChainPosteriorEqualsPrior:Eq:ProjectionProbabilityIdentity}), we see the following two facts

\begin{itemize}
\item[Fact 1':] For each of the different sequences $\mathscr{S}$ such that $\mathfrak{Y}(\mathscr{S})$ is a string of chain type with head $i$, $\mathscr{P}(\mathscr{S})$ has the same value.
\item[Fact 2':] The sequence $\mathscr{S}$ in the RHS of (\ref{Prop:MarkovChainPosteriorEqualsPrior:Eq:ProjectionProbabilityIdentity}) can be an arbitrary length-$n$ sequence picked from all the possible choices of $\mathscr{S}$ such that $\mathfrak{Y}(\mathscr{S})$ is a string of chain type with head $i$.
\end{itemize}

For each string of chain type $((i, i_2) , ..., (i_{n-1}, i_n))$ with head $i$, we collect all possible sequences $\mathscr{S}$ such that $\mathfrak{Y}(\mathscr{S})=((i, i_2) , ..., (i_{n-1}, i_n))$. We claim that we have 

\begin{equation}\label{Prop:MarkovChainPosteriorEqualsPrior:Eq:ProjectionProbabilityIdentityUnionS}
\begin{array}{ll}
& \mathbf{P}(X_1=(i, i_2), ..., X_n=(i_{n-1}, i_n)|\mathcal{E}_{\{\nu_{(i,j)}\}}, Y_1=i)
\\
=& \mathscr{P}\left(\text{all possible }\mathscr{S} \text{ such that } \mathfrak{Y}(\mathscr{S})=((i, i_2) , ..., (i_{n-1}, i_n))\right) \ .
\end{array}
\end{equation}

This is because due to the above Fact 1', we have
$$\begin{array}{ll}
& \mathscr{P}\left(\text{all possible }\mathscr{S} \text{ such that } \mathfrak{Y}(\mathscr{S})=((i, i_2) , ..., (i_{n-1}, i_n))\right)
\\
= & \left(\text{Number of all possible } \mathscr{S} \text{ such that } \mathfrak{Y}(\mathscr{S})=((i, i_2) , ..., (i_{n-1}, i_n))\right)\cdot \mathscr{P}(\mathscr{S}) \ .
\end{array}$$

Now we note that by simple combinatorics we have

\begin{itemize}
\item[Fact 3':] For each realization $(i, i_2) , ..., (i_{n-1}, i_n)$ of $X_1,...,X_n$,
$$\left(\text{Number of all possible } \mathscr{S} \text{ such that } \mathfrak{Y}(\mathscr{S})=((i, i_2) , ..., (i_{n-1}, i_n))\right)$$
is independent of the choice of $(i, i_2) , ..., (i_{n-1}, i_n)$\footnote{This number is actually 
$\nu_{(i_1, j_1)}! \, \nu_{(i_2, j_2)}! \, ... \, \nu_{(i_I, j_I)}!$ \ .}. 
\end{itemize}

Therefore by Fact 3' we see that (\ref{Prop:MarkovChainPosteriorEqualsPrior:Eq:ProjectionProbabilityIdentityUnionS}) is equivalent to 

$$\mathbf{P}(X_1=(i, i_2),...,X_n=(i_{n-1},i_n)|\mathcal{E}_{\{\nu_{(i,j)}\}}, Y_1=i)=K_1(i)\mathscr{P}(\mathscr{S}) \ .$$
It is then easy to see that $K_1(i)=K(i)$ just by normalization of the probability measure $\mathscr{P}(\bullet)$ and the conditional probability measure $\mathbf{P}(\bullet|\mathcal{E}_{\{\nu_{(i,j)}\}}, Y_1=i)$, as well as the above Fact 2'. So we proved that (\ref{Prop:MarkovChainPosteriorEqualsPrior:Eq:ProjectionProbabilityIdentityUnionS}) is valid.

From (\ref{Prop:MarkovChainPosteriorEqualsPrior:Eq:ProjectionProbabilityIdentityUnionS}) we know that for an element 
$(i,j)\in \{(i_1, j_1),...,(i_I, j_I)\}$ with the first component $i$ being fixed, we have

\begin{equation}\label{Prop:MarkovChainPosteriorEqualsPrior:Eq:PosterialMarginalDecomposition:Step:Counting}
\mathbf{P}(X_1=(i,j)|\mathcal{E}_{\{\nu_{(i,j)}\}}, Y_1=i)
= \mathbf{1}^{j, \checkmark}_2
\cdot \sum\limits_{a=1}^{\nu_{(i,j)}} \mathscr{P}\left(\text{all possible }\mathscr{S} \text{ whose 1-st element is } ((i,j),a)\right) \ .
\end{equation}

We claim that for each element $((i,j),a)\in \mathcal{X}^{\text{distinguished}}$ with $i$ being fixed we also have

\begin{itemize}
\item[Fact 4':] For each element $((i,j),a)\in \mathcal{X}^{\text{distinguished}}$ with $i$ being fixed, 
$$\left(\text{Number of all possible } \mathscr{S} \text{ whose 1-st element is } ((i,j),a)\right)$$
is independent of $a$, and is actually equal to 
$\dfrac{\nu_{(i_1, j_1)}! \, ... \, \nu_{(i_I, j_I)}!}{\nu_{(i,j)}}\cdot \mathbf{1}^{j,\checkmark}_2\cdot\text{\#}_1^{(i,j)}(\mathcal{E}_{\{\nu_{(i,j)}\}})$. 
\item[Fact 5':] $\mathscr{P}\left(\text{all possible }\mathscr{S} \text{ whose 1-st element is } ((i,j),a)\right)=\dfrac{\nu_{(i_1, j_1)}! \, ... \, \nu_{(i_I, j_I)}!}{\nu_{(i,j)}}\cdot \mathbf{1}^{j,\checkmark}_2\cdot \text{\#}_1^{(i,j)}(\mathcal{E}_{\{\nu_{(i,j)}\}})\cdot p$, where $p$ is independent of $((i,j),a)$.
\end{itemize}

The above Fact 4' is a simple combinatorial observation, and Fact 5' is a consequence of Facts 1' and 4'. Since we have
$$\sum\limits_{((i,j),a)\in \mathcal{X}^{\text{distinguished}}, i \text{ fixed}}
\mathscr{P}\left(\text{all possible }\mathscr{S} \text{ whose 1-st element is } ((i,j),a)\right)=1 \ ,$$
by Facts 4' and 5' we see that 
$$1=\sum\limits_{k_2=1}^N \mathbf{1}^{k_2,\checkmark}_2\cdot
\nu_{(i,k_2)}\cdot \dfrac{\nu_{(i_1, j_1)}! \, ... \, \nu_{(i_I, j_I)}!}{\nu_{(i,k_2)}}\cdot \text{\#}_1^{(i,k_2)}(\mathcal{E}_{\{\nu_{(i,k_2)}\}})\cdot p \ ,$$
so 
$$p=\dfrac{1}{\sum_{k_2=1}^N \mathbf{1}^{k_2,\checkmark}_2\cdot \nu_{(i_1, j_1)}! \, ... \, \nu_{(i_I, j_I)}! \cdot \text{\#}_1^{(i,k_2)}(\mathcal{E}_{\{\nu_{(i,k_2)}\}})} \ .$$
By the above equation, (\ref{Prop:MarkovChainPosteriorEqualsPrior:Eq:PosterialMarginalDecomposition:Step:Counting}) and Fact 5' we see that (\ref{Prop:MarkovChainPosteriorEqualsPrior:Eq}) holds when $\mathcal{E}_{\{\nu_{(i,j)}\}}$ is not disjoint with $\{Y_1=i\}$.

Finally we see that if $\mathcal{E}_{\{\nu_{(i,j)}\}}$ is disjoint with $\{Y_1=i\}$, then $\mathbf{1}^{i,\checkmark}_1=0$ and (\ref{Prop:MarkovChainPosteriorEqualsPrior:Eq}) is automatically true. So we have proved the whole statement.
\end{proof}

\begin{remark}[Conditional Symmetry]\rm\label{Remark:MarkovChainProofIllustratesConditionalSymmetry}
The above argument is parallel to the proof of Theorem \ref{Thm:IIDPosteriorEqualsPrior}. Here, as is the same in Remark \ref{Remark:IIDProofIllustratesConditionalSymmetry}, we used the idea of extending the probability space and our obtained identity (\ref{Prop:MarkovChainPosteriorEqualsPrior:Eq:ProjectionProbabilityIdentityUnionS}) is again relating the conditional probability to an absolute probability. Fact 1' is parallel to Fact 1, and is a manifestation of the conditional symmetry in the Markov chain case.

However, in the Markov chain case, we do not have the combinatorial symmetry in exactly the same way as Fact 4 proposed in the proof of Theorem \ref{Thm:IIDPosteriorEqualsPrior}. Rather, this is replaced by our new Fact 4', which leads to Fact 5'. This is the main reason why we cannot simply use the frequency formula 
\begin{equation}\label{Remark:Eq:FrenquencyFormula}
\mathbf{P}\big(X_1=(i,j) \big|\mathcal{E}_{\{\nu_{(i,j)}\}}, Y_1=i\big) = 
\dfrac{\nu_{(i,j)}}
{\sum_{k_2=1}^N \nu_{(i, k_2)}}
\end{equation}
to replace (\ref{Prop:MarkovChainPosteriorEqualsPrior:Eq}). However, we claim that asymptotically it is still true to use (\ref{Remark:Eq:FrenquencyFormula}) to replace our obtained formula (\ref{Prop:MarkovChainPosteriorEqualsPrior:Eq}). This issue will be further investigated in the next Section.

Again, in a same fashion as Remark \ref{Remark:AlternateProof:Prop:IIDPosteriorEqualsPrior}, our result for the Markov chain case can also be proved directly using the fact that the joint distribution of $X_1,...,X_n$ remains the same conditioned on $\mathcal{E}_{\nu_{(i,j)}}$ and $Y_1=i$ regardless of how we place the outcomes $(i, i_1), ..., (i_{n-1}, i_n)$. In our proof, the main purpose of introducing an extended probability space is to show that the combinatorial symmetry will be broken into Fact 5', so that we cannot simply use the observed frequencies to calculate the posterior marginal distribution.
\end{remark}

\begin{proposition}\label{Prop:MarkovChainPosteriorEqualsPrior:Y_1}
Given $1\leq i\leq N$, then we have
\begin{equation}\label{Prop:MarkovChainPosteriorEqualsPrior:Y_1:Eq}
\mathbf{P}(Y_1=i|\mathcal{E}_{\{\nu_{(i,j)}\}})=\dfrac{\mathbf{1}^{i, \checkmark}_1 \pi^0_i}{\sum_{k_1=1}^N \mathbf{1}^{k_1, \checkmark}_1 \pi^0_{k_1}} \ .
\end{equation}
\end{proposition}

\begin{proof}
We first consider the case when the events $\{Y_1=i\}$ and $\mathcal{E}_{\{\nu_{(i,j)}\}}$ are not disjoint. In this case, there must be a pair $(i,j)$ with the first component $i$ being fixed, such that $\nu_{(i,j)}\geq 1$. Thus $\mathcal{E}_{\{\nu_{(i,j)}\}}\subseteq \{Y_1=i\}$ and we have
$$\mathbf{P}(Y_1=i|\mathcal{E}_{\{\nu_{(i,j)}\}})
= \dfrac{1}{\mathbf{P}(\mathcal{E}_{\{\nu_{(i,j)}\}})}\mathbf{P}(Y_1=i, \mathcal{E}_{\{\nu_{(i,j)}\}})
=  C \pi^0_i \ ,
$$
where $C>0$ is a normalizing constant. 

It is easy to see that $\mathbf{1}^{i, \checkmark}_1=0$ if and only if the events $\{Y_1=i\}$ and $\mathcal{E}_{\{\nu_{(i,j)}\}}$ are disjoint. Thus in general we have
$$\mathbf{P}(Y_1=i|\mathcal{E}_{\{\nu_{(i,j)}\}})
=  C \mathbf{1}^{i, \checkmark}_1 \pi^0_i \ .
$$
By normalization of the conditional probability $\mathbf{P}(\bullet|\mathcal{E}_{\{\nu_{(i,j)}\}})$ we know that the normalization constant $C=\dfrac{1}{\sum_{k_1=1}^N \mathbf{1}^{k_1, \checkmark}_1 \pi^0_{k_1}}$, so we proved the statememt of the Proposition.
\end{proof}

Combining Propositions \ref{Prop:MarkovChainPosteriorEqualsPrior}, \ref{Prop:MarkovChainPosteriorEqualsPrior:Y_1} we easily have

\begin{theorem}[posterior distribution for the finite Markov chain case]\label{Thm:Prop:MarkovChainPosteriorEqualsPrior}
Given $1\leq i,j\leq N$, then we have
\emph{\begin{equation}\label{Thm:Prop:MarkovChainPosteriorEqualsPrior:Eq}
\mathbf{P}\big(X_1=(i,j) \big|\mathcal{E}_{\{\nu_{(i,j)}\}}\big) = 
\dfrac{\mathbf{1}^{i, \checkmark}_1 \pi^0_i}{\sum_{k_1=1}^N \mathbf{1}^{k_1, \checkmark}_1 \pi^0_{k_1}}
\cdot \dfrac{\mathbf{1}^{j, \checkmark}_2 \cdot \text{\#}^{(i,j)}_1(\mathcal{E}_{\{\nu_{(i,j)}\}})}
{\sum_{k_2=1}^N \mathbf{1}^{k_2, \checkmark}_2 \cdot 
\text{\#}^{(i,k_2)}_1(\mathcal{E}_{\{\nu_{(i,j)}\}})}
\ .
\end{equation}}
\end{theorem}

\begin{proof}
This is an easy consequence of the simple conditional probability formula
\begin{equation}\label{Thm:Prop:MarkovChainPosteriorEqualsPrior:Eq:ConditionalProbabilityFormula}
\mathbf{P}\big(X_1=(i,j) \big|\mathcal{E}_{\{\nu_{(i,j)}\}}\big) = 
\mathbf{P}(Y_1=i|\mathcal{E}_{\{\nu_{(i,j)}\}})
\cdot \mathbf{P}\big(X_1=(i,j) \big|\mathcal{E}_{\{\nu_{(i,j)}\}}, Y_1=i\big) \ ,
\end{equation}
as well as the fact that $X_1=(i,j)$ implies $Y_1=i$.
\end{proof}

For the rest of this section, we would like to focus more on the combinatorial calculation of the number of strings of chain type $\text{\#}^{(i,j)}_1(\mathcal{E}_{\{\nu_{(i,j)}\}})$ that we have introduced in Definition \ref{Def:NumberStringsGivenTerm}. Actually, this number is calculated based on the so-called Whittle's formula (see \cite{whittle1955}, also \cite[Theorem 2.1]{billingsly1961AnnStat}). For the reader's convenience, we shall first describe Whittle's result below. Our formulation of this result is based on \cite[Theorem 2.1]{billingsly1961AnnStat}, but the mathematical terms and symbols we use will follow those in our present paper.

Given a frequency event $\mathcal{E}_{\{\nu_{(i,j)}\}}$ and a string of chain type $X_1=(i_1, i_2),...,X_n=(i_n, i_{n+1})$ that satisfies $\mathcal{E}_{\{\nu_{(i,j)}\}}$, it is easy to observe that 

\begin{equation}\label{Eq:StringOfChainTypeFrequencyBalance}
\sum\limits_{j=1}^{N} \nu_{(i,j)} - \sum\limits_{j=1}^{N} \nu_{(j,i)}  =\mathbf{1}_{\{i=i_1\}}-\mathbf{1}_{\{i=i_{n+1}\}} \ .
\end{equation}

It is easy to see that once $\mathcal{E}_{\{\nu_{(i,j)}\}}$ is given, then $i_1$ is fixed if we fix $i_{n+1}$, and $i_{n+1}$ is fixed if we fix $i_1$. However, if we do not fix either $i_1$ or $i_{n+1}$, then we may have different choices of both of them. As an example, the strings of chain type $(1,2), (2,1)$ and $(2,1), (1,2)$ correspond to the same frequency event $\mathcal{E}_{\{\nu_{(i,j)}\}}$, but the choices of $i_1$ and $i_{n+1}$ can be different.

Recall we have assumed that the state space of the chain $\{Y_{\ell}\}_{\ell\geq 1}$ is a finite set $\Sigma=\{1,2,...,N\}$. Let us fix some $u, v\in \{1,2,...,N\}$ and consider all possible strings of chain type $X_1=(i_1, i_2),...,X_n=(i_n, i_{n+1})$ that satisfy the given frequency event $\mathcal{E}_{\{\nu_{(i,j)}\}}$, such that $i_1=u, i_{n+1}=v$. The total number of such strings of chain type is denoted by $N_{uv}^{(n)}(\mathcal{E}_{\{\nu_{(i,j)}\}})$. We shall first form a matrix $F^*$ of size $N\times N$ with elements $F^*=(\nu_{ij}^*)_{1\leq i, j\leq N}$, where
\begin{equation}\label{Eq:WhittleFormulaF*} 
\nu_{ij}^*=\left\{
\begin{array}{lll}
\mathbf{1}_{\{i=j\}}-\dfrac{\nu_{(i,j)}}{\sum\limits_{j=1}^{N} \nu_{(i,j)}} \ , & \text{ if } & \sum\limits_{j=1}^{N} \nu_{(i,j)}>0 \ ,  
\\
\mathbf{1}_{\{i=j\}} \ , & \text{ if } & \sum\limits_{j=1}^{N} \nu_{(i,j)}=0 \ .
\end{array}\right.
\end{equation}

The following result is due to Whittle in 1955 (see \cite{whittle1955}) and its proof is also presented in \cite[Theorem 2.1]{billingsly1961AnnStat}.

\begin{theorem}[Whittle's formula, 1955]\label{Thm:WhittleFormula}
We have
\begin{equation}\label{Thm:WhittleFormula:Eq}
N_{uv}^{(n)}(\mathcal{E}_{\{\nu_{(i,j)}\}})=\dfrac{\prod\limits_{i=1}^{N} \left(\sum\limits_{j=1}^{N} \nu_{(i,j)}\right)!}
{\prod\limits_{i=1}^{N}\prod\limits_{j=1}^{N} \nu_{(i,j)}!}F_{vu}^* \ ,
\end{equation}
where $F_{vu}^*$ is the $(v,u)$-th cofactor of the matrix $F^*$ and $0!=1$.
\end{theorem}

Very simply, the Whittle's formula provides us with an easy calculation of the quantity $\text{\#}^{(i,j)}_1(\mathcal{E}_{\{\nu_{(i,j)}\}})$ that we have introduced in Definition \ref{Def:NumberStringsGivenTerm}. To this end, for given $i,j\in \{1,2,...,N\}$ we define the $N\times N$ matrix $\widetilde{F}^*=(\widetilde{\nu}^*_{\widetilde{i}\,\widetilde{j}})_{1\leq \widetilde{i},\widetilde{j}\leq N}$, where

\begin{equation}\label{Eq:WhittleFormulaF*Tilde} 
\widetilde{\nu}_{\widetilde{i} \, \widetilde{j}}^*=\left\{
\begin{array}{lll}
\mathbf{1}_{\{i=j\}}-\dfrac{\nu_{(i,j)}-1}{\sum\limits_{k=1}^{N} \nu_{(i,k)}-1} \ , & \text{ if } & \widetilde{i}=i , \widetilde{j}=j \text{ and} \sum\limits_{k=1}^{N} \nu_{(i,k)}>1 \ , \ \nu_{(i,j)}\geq 1 \ ,
\\
\mathbf{1}_{\{i=\widetilde{j}\}}-\dfrac{\nu_{(i,\widetilde{j})}}{\sum\limits_{k=1}^{N} \nu_{(i,k)}-1} \ , & \text{ if } & \widetilde{i}=i, \widetilde{j}\neq j \text{ and} \sum\limits_{k=1}^{N} \nu_{(i,k)}>1 \ , \ \nu_{(i,j)}\geq 1 \ ,
\\
\mathbf{1}_{\{i=\widetilde{j}\}} \ , & \text{ if } & \widetilde{i}=i \text{ and} \sum\limits_{k=1}^{N} \nu_{(i,k)}=1 \ , \ \nu_{(i,j)}= 1 \ ,
\\
\mathbf{1}_{\{\widetilde{i}=\widetilde{j}\}}-\dfrac{\nu_{(\widetilde{i},\widetilde{j})}}{\sum\limits_{k=1}^{N} \nu_{(\widetilde{i},k)}} \ , & \text{ if } & \widetilde{i}\neq i \text{ and} \sum\limits_{k=1}^{N} \nu_{(\widetilde{i},k)}>0 \ ,
\\
\mathbf{1}_{\{\widetilde{i}=\widetilde{j}\}} \ , & \text{ if } & \widetilde{i}\neq i \text{ and} \sum\limits_{k=1}^{N} \nu_{(\widetilde{i},k)}=0 \ .
\end{array}\right.
\end{equation}

Then we have
\begin{corollary}[The exact calculation of the number of strings of chain type with given term]\label{Corollary:NumberOfChainsStringType} Given the frequency event $\mathcal{E}_{\{\nu_{(i,j)}\}}$ that satisfies \emph{(\ref{Eq:StringOfChainTypeFrequencyBalance})} with fixed $i_1=i$ and $i_{n+1}=v$, and suppose that $\nu_{(i,j)}\geq 1$ for fixed $j$. Then the quantity \emph{$\text{\#}^{(i,j)}_1(\mathcal{E}_{\{\nu_{(i,j)}\}})$} that we have introduced in \emph{Definition \ref{Def:NumberStringsGivenTerm}} can be calculated as
\begin{equation}\label{Corollary:NumberOfChainsStringType:Eq}
\emph{\text{\#}}^{(i,j)}_1(\mathcal{E}_{\{\nu_{(i,j)}\}})=
\left\{
\begin{array}{ll}
\dfrac{\nu_{(i,j)}}{\sum\limits_{j=1}^N \nu_{(i,j)}}\cdot N_{iv}^{(n)}(\mathcal{E}_{\{\nu_{(i,j)}\}})\cdot 
\dfrac{\widetilde{F}_{vj}^*}{F_{vi}^*}  \ , & \emph{\text{ if }} N_{iv}^{(n)}(\mathcal{E}_{\{\nu_{(i,j)}\}}) > 0 \ ;
\\
0 \ , & \emph{\text{ if }} N_{iv}^{(n)}(\mathcal{E}_{\{\nu_{(i,j)}\}}) = 0 \ ,
\end{array}
\right. 
\end{equation}
where $N_{iv}^{(n)}(\mathcal{E}_{\{\nu_{(i,j)}\}})$ is the quantity defined in the Whittle's formula \emph{(\ref{Thm:WhittleFormula:Eq})} with $u=i, v=v$; 
$F^*_{vi}$ is the $(v,i)$-th cofactor of the matrix $F^*=(\nu_{ij})_{1\leq i,j\leq N}$ with elements defined in \emph{(\ref{Eq:WhittleFormulaF*})}, and $\widetilde{F}^*_{vj}$ is the $(v,j)$-th cofactor of the matrix $\widetilde{F}^*=(\widetilde{\nu}_{(\widetilde{i} \, \widetilde{j})})_{1\leq \widetilde{i} \, \widetilde{j} \leq N}$ with elements defined in \emph{(\ref{Eq:WhittleFormulaF*Tilde})}. 
\end{corollary}

\begin{proof}
Given $\mathcal{E}_{\{\nu_{(i,j)}\}}$ and some fixed $i, j\in \{1,2,...,N\}$, let us first suppose that at least one string of chain type $X_1=(i, j), X_2=(j, i_3), ..., X_n=(i_n, i_{n+1})$ that satisfies $\mathcal{E}_{\{\nu_{(i,j)}\}}$ exists, which implies that $\text{\#}^{(i,j)}_1(\mathcal{E}_{\{\nu_{(i,j)}\}})\geq 1$. It is easy to observe that the sequence $X_2=(j, i_3), ..., X_n=(i_n, i_{n+1})$ forms a string of chain type that satisfies the frequency event $\mathcal{E}_{\{\widetilde{\nu}_{(\widetilde{i},\widetilde{j})}\}}$ with 
\begin{equation}\label{Corollary:NumberOfChainsStringType:Eq:ModifiedFrequencies}
\widetilde{\nu}_{(\widetilde{i},\widetilde{j})}=
\left\{\begin{array}{ll}
\nu_{(\widetilde{i}, \widetilde{j})} \ , & \text{ if } (\widetilde{i}, \widetilde{j}) \neq (i,j) \ ,
\\
\nu_{(\widetilde{i}, \widetilde{j})}-1 \ , & \text{ if } (\widetilde{i}, \widetilde{j}) = (i,j) \ .
\end{array}\right.
\end{equation}

We note that, when $\mathcal{E}_{\{\nu_{(i,j)}\}}$ and the starting state $i$ are fixed, then the final state $i_{n+1}=v\in \{1,2,...,N\}$ will be fixed, and thus $N_{iv}^{(n)}(\mathcal{E}_{\nu_{(i,j)}})\geq 1$. Let us also note that $\mathcal{E}_{\{\widetilde{\nu}_{(\widetilde{i}, \widetilde{j})}\}}$ must satisfy (\ref{Eq:StringOfChainTypeFrequencyBalance}) in the form of 
$$\sum\limits_{\widetilde{j}=1}^{N} \widetilde{\nu}_{(\widetilde{i},\widetilde{j})} - \sum\limits_{\widetilde{j}=1}^{N} \widetilde{\nu}_{(\widetilde{j},\widetilde{i})}  =\mathbf{1}_{\{\widetilde{i}=j\}}-\mathbf{1}_{\{\widetilde{i}=v\}} \ .
$$
This means that given $\mathcal{E}_{\{\widetilde{\nu}_{(\widetilde{i}, \widetilde{j})}\}}$ and the starting state $j$, the final state must be $v$. Thus it is easy to see that
\begin{equation}\label{Corollary:NumberOfChainsStringType:Eq:NumberOfAsmissiblePathsDecomposition}  
\#_1^{(i,j)}(\mathcal{E}_{\nu_{(i,j)}})=N_{jv}^{(n-1)}(\mathcal{E}_{\{\widetilde{\nu}_{(\widetilde{i},\widetilde{j})}\}}) \ .
\end{equation}

According to Theorem \ref{Thm:WhittleFormula}, the RHS of (\ref{Corollary:NumberOfChainsStringType:Eq:NumberOfAsmissiblePathsDecomposition}) is equal to 
\begin{equation}\label{Corollary:NumberOfChainsStringType:Eq:ModifiedNumberAdmissiblePaths} 
N_{jv}^{(n-1)}(\mathcal{E}_{\{\widetilde{\nu}_{(\widetilde{i},\widetilde{j})}\}})=
\dfrac{\prod\limits_{\widetilde{i}=1}^{N} \left(\sum\limits_{\widetilde{j}=1}^{N} \widetilde{\nu}_{(\widetilde{i},\widetilde{j})}\right)!}
{\prod\limits_{\widetilde{i}=1}^{N}\prod\limits_{\widetilde{j}=1}^{N} \widetilde{\nu}_{(\widetilde{i},\widetilde{j})}!} \widetilde{F}_{vj}^* \ ,
\end{equation}
where $\widetilde{F}_{vj}^*$ is the $(v,j)$-th cofactor of the matrix $\widetilde{F}^*=(\widetilde{\nu}^*_{\widetilde{i} \, \widetilde{j}})_{1\leq \widetilde{i}, \widetilde{j}\leq N}$, with

\begin{equation}\label{Corollary:NumberOfChainsStringType:Eq:WhittleFormulaF*Tilde} 
\widetilde{\nu}_{\widetilde{i} \, \widetilde{j}}^*=\left\{
\begin{array}{lll}
\mathbf{1}_{\{\widetilde{i}=\widetilde{j}\}}-\dfrac{\widetilde{\nu}_{(\widetilde{i},\widetilde{j})}}{\sum\limits_{\widetilde{j}=1}^{N} \widetilde{\nu}_{(\widetilde{i}, \widetilde{j})}} \ , & \text{ if } & \sum\limits_{\widetilde{j}=1}^{N} \widetilde{\nu}_{(\widetilde{i},\widetilde{j})}>0 \ ,  
\\
\mathbf{1}_{\{\widetilde{i}=\widetilde{j}\}} \ , & \text{ if } & \sum\limits_{\widetilde{j}=1}^{N} \widetilde{\nu}_{(\widetilde{i},\widetilde{j})}=0 \ .
\end{array}\right.
\end{equation}
It is then easy to see that (\ref{Corollary:NumberOfChainsStringType:Eq:WhittleFormulaF*Tilde}) is the same as (\ref{Eq:WhittleFormulaF*Tilde}) due to (\ref{Corollary:NumberOfChainsStringType:Eq:ModifiedFrequencies}).

By (\ref{Corollary:NumberOfChainsStringType:Eq:ModifiedFrequencies}), we also know that
\begin{equation}\label{Corollary:NumberOfChainsStringType:Eq:ModifiedFrequenciesSum1}
\sum\limits_{\widetilde{j}=1}^N \widetilde{\nu}_{(\widetilde{i}, \widetilde{j})}
=
\left\{\begin{array}{ll}
\sum\limits_{j=1}^N \nu_{(i,j)} \ , & \text{ if } \widetilde{i}\neq i \ ,
\\
\sum\limits_{j=1}^N \nu_{(i,j)}-1 \ , & \text{ if } \widetilde{i}=i \ .
\end{array}\right.
\end{equation}
Therefore
\begin{equation}\label{Corollary:NumberOfChainsStringType:Eq:ModifiedFrequenciesSum2}
\prod\limits_{\widetilde{i}=1}^{N}
\left(\sum\limits_{\widetilde{j}=1}^N 
\widetilde{\nu}_{(\widetilde{i}, \widetilde{j})}\right)!
=
\dfrac{\prod\limits_{i=1}^{N}
\left(\sum\limits_{j=1}^N 
\nu_{(i, j)}\right)!}
{\sum\limits_{j=1}^N 
\nu_{(i, j)}} \ .
\end{equation}
Moreover, by (\ref{Corollary:NumberOfChainsStringType:Eq:ModifiedFrequencies}) we know that 
\begin{equation}\label{Corollary:NumberOfChainsStringType:Eq:ModifiedFrequenciesSum3}
\prod\limits_{\widetilde{i}=1}^{N}\prod\limits_{\widetilde{j}=1}^{N} \widetilde{\nu}_{(\widetilde{i},\widetilde{j})}!
=
\dfrac{\prod\limits_{i=1}^{N} \prod\limits_{j=1}^{N} \nu_{(i,j)}!}{\nu_{(i,j)}} \ .
\end{equation}

Putting (\ref{Corollary:NumberOfChainsStringType:Eq:ModifiedFrequenciesSum2}), (\ref{Corollary:NumberOfChainsStringType:Eq:ModifiedFrequenciesSum3}) into (\ref{Corollary:NumberOfChainsStringType:Eq:ModifiedNumberAdmissiblePaths}) and making use of (\ref{Thm:WhittleFormula:Eq}) we know that 
\begin{equation}\label{Corollary:NumberOfChainsStringType:Eq:ModifiedNumberAdmissiblePathsFinalForm}  
N_{jv}^{(n-1)}(\mathcal{E}_{\{\widetilde{\nu}_{(\widetilde{i},\widetilde{j})}\}})=\dfrac{\nu_{(i,j)}}
{\sum\limits_{j=1}^N \nu_{(i,j)}}\cdot 
N_{iv}^{(n)}(\mathcal{E}_{\{\nu_{(i,j)}\}})\cdot 
\dfrac{\widetilde{F}_{vj}^*}{F_{vi}^*} \ .
\end{equation} 
where $F^*_{vi}$ is the $(v,i)$-th cofactor of the matrix $F^*=(\nu_{ij})_{1\leq i,j\leq N}$ with elements defined in (\ref{Eq:WhittleFormulaF*}); $\widetilde{F}^*_{vj}$ is the $(v,j)$-th cofactor of the matrix $\widetilde{F}^*=(\widetilde{\nu}_{(\widetilde{i} \, \widetilde{j})})_{1\leq \widetilde{i} \, \widetilde{j} \leq N}$ with elements defined in (\ref{Corollary:NumberOfChainsStringType:Eq:WhittleFormulaF*Tilde}); and $N_{iv}^{(n)}(\mathcal{E}_{\{\nu_{(i,j)}\}})$ is the quantity defined in the Whittle's formula (\ref{Thm:WhittleFormula:Eq}) with $u=i, v=v$. 
Combining (\ref{Corollary:NumberOfChainsStringType:Eq:NumberOfAsmissiblePathsDecomposition}) and (\ref{Corollary:NumberOfChainsStringType:Eq:ModifiedNumberAdmissiblePathsFinalForm}) we get (\ref{Corollary:NumberOfChainsStringType:Eq}) as desired.

Finally, let us consider the case when no string of chain type $X_1=(i, j), X_2=(j, i_3), ..., X_n=(i_n, i_{n+1})$ that satisfies $\mathcal{E}_{\{\nu_{(i,j)}\}}$ exists, which is the same as saying that $\text{\#}^{(i,j)}_1(\mathcal{E}_{\{\nu_{(i,j)}\}})=0$. In this case, it is easy to see that $N_{jv}^{(n-1)}(\mathcal{E}_{\{\widetilde{\nu}_{(\widetilde{i},\widetilde{j})}\}})=0$. If  $N_{iv}^{(n)}(\mathcal{E}_{\{\nu_{(i,j)}\}})\geq 1$, then by (\ref{Corollary:NumberOfChainsStringType:Eq:ModifiedNumberAdmissiblePathsFinalForm}) and the Whittle's formula for $N_{iv}^{(n)}(\mathcal{E}_{\{\nu_{(i,j)}\}})$ we know that $\widetilde{F}^*_{vj}=0$ and $F^*_{vi}>0$, which indicates that (\ref{Corollary:NumberOfChainsStringType:Eq}) is still correct. If $N_{iv}^{(n)}(\mathcal{E}_{\{\nu_{(i,j)}\}})=0$, we apply the second part of (\ref{Corollary:NumberOfChainsStringType:Eq}) and we know that it is still correct. 
\end{proof}

We provide here another technical Lemma regarding the $(v,u)$-cofactor $F^*_{vu}$ of the matrix $F^*=(\nu^*_{ij})_{1\leq i,j\leq N}$ defined in (\ref{Eq:WhittleFormulaF*}). A version of the same Lemma can be found in \cite[Lemma 4]{IEEEConditioning}. For comprehensiveness, we will also provide a short proof of this Lemma.

\begin{lemma}\label{Lm:WhittleFormulaF*CofactorSame}
Suppose that for each $1\leq v\leq N$ we have $\sum\limits_{k=1}^N \nu_{(v,k)}>0$. Then for each $1\leq v\leq N$ and each pair $1\leq i,j\leq N$ we have $F^*_{vi}=F^*_{vj}$.
\end{lemma}

\begin{proof}
Recall that the matrix $F^*=(\nu_{ij}^*)_{1\leq i, j\leq N}$ has the matrix elements $\nu^*_{ij}$ that we defined in (\ref{Eq:WhittleFormulaF*}). Under our assumption we find that for each $1\leq v\leq N$ we have
\begin{equation}\label{Lm:WhittleFormulaF*CofactorSame:Eq:F*RowSum0}
\sum\limits_{l=1}^N \nu^*_{vl} = 0 \ .
\end{equation}
To prove the statement of this lemma, without loss of generality we can assume that $v=1$, $i=1$ and $j=2$. Set the matrices
$$\mathcal{F}^*_{11}=\begin{pmatrix}
\nu_{22}^* & \nu_{23}^* & ... & \nu_{2N}^*
\\
\nu_{32}^* & \nu_{33}^* & ... & \nu_{3N}^*
\\
...
\\
\nu_{N2}^* & \nu_{N3}^* & ... & \nu_{NN}^*
\end{pmatrix} 
\ , \
\mathcal{F}^*_{12}=\begin{pmatrix}
\nu_{21}^* & \nu_{23}^* & ... & \nu_{2N}^*
\\
\nu_{31}^* & \nu_{33}^* & ... & \nu_{3N}^*
\\
...
\\
\nu_{N1}^* & \nu_{N3}^* & ... & \nu_{NN}^*
\end{pmatrix} \ .$$
Then $F^*_{11}=\det\mathcal{F}^*_{11}$ and $F^*_{12}=-\det \mathcal{F}^*_{12}$. It is then easy to see that by adding each of the second, third, ... , until the last columns of the matrix $\mathcal{F}^*_{11}$ to its first column, and using (\ref{Lm:WhittleFormulaF*CofactorSame:Eq:F*RowSum0}), we arrive at the matrix $-\mathcal{F}_{12}^*$. So we conclude that $F_{11}^*=F_{12}^*$. For any other $1\leq i, j \leq N$, the argument is the same. Thus the Lemma is proved. 
\end{proof}

\section{The infinite-sample limit of the Markov chain case}\label{Sec:InfiniteTimeLimitMarkovChain}

In this section, by making use of Theorem \ref{Thm:Prop:MarkovChainPosteriorEqualsPrior}, we tend to investigate the limit of the posterior probability $\mathbf{P}\big(X_1=(i,j) \big|\mathcal{E}_{\{\nu_{(i,j)}\}}\big)$ as the number of observations $n$ tends to infinity. Recall that we have introduced in Section \ref{Sec:MarkovChain} the Markov chain $\{Y_\ell\}_{\ell\geq 1}$ with finite state space $\Sigma=\{1,...,N\}$, $|\Sigma|=N$. Let the transition probability matrix of the process $\{Y_\ell\}_{\ell\geq 1}$ be given by $P=(p_{ij})_{1\leq i, j\leq N}$.
Assume the process starts from an initial probability distribution $\pi^0=(\pi^0_1,...,\pi^0_N)$, $0\leq \pi^0_i\leq 1$, $\sum\limits_{i=1}^N \pi^0_i=1$, such that $\mathbf{P}(Y_1=i)=\pi^0_i$.  Within this section, we will work under the following

\begin{assumption}[Positive Transition Probabilities]\label{Assmption:PositiveTransition}
The Markov chain $\{Y_\ell\}_{\ell\geq 1}$ has positive transition probabilities, i.e. each transition probability $p_{ij}>0$ for any $1\leq i, j\leq N$. 
Thus the Markov chain $\{Y_{\ell}\}_{\ell\geq 1}$ is ergodic, and its invariant measure is given by $\pi_i>0$, $i=1,2,...,N$ with $\sum_{i=1}^N \pi_i=1$. 
\end{assumption}

Let us fix the number of states $N$ and the number of observations $n$, so that we consider the Markov chain $Y_1,...,Y_n,Y_{n+1}$ and the ``consecutive pair" process $X_\ell=(Y_\ell, Y_{\ell+1})$ for $1\leq \ell \leq n+1$ as in Section \ref{Sec:MarkovChain}. Under this assumption, the number of trajectories $(Y_1,...,Y_{n+1})$ is finite and is equal to $N^{n+1}$. Since each trajectory corresponds to a frequency event $\mathcal{E}_{\{\nu_{(i,j)}\}}$ as defined in (\ref{Eq:MarkovChainSampleFrequencyEvent}), and two trajectories may correspond to the same frequency event \footnote{For example: when the chain $\{Y_\ell\}_{\ell \geq 1}$ is given by $1\rightarrow 1 \rightarrow 1 \rightarrow 2 \rightarrow 1$ or $1\rightarrow 2 \rightarrow 1 \rightarrow 1 \rightarrow 1$, both cases will correspond to $\mathcal{E}_{\{\nu_{(i,j)}\}}=\{\nu_{(1,2)}=\nu_{(2,1)}=1, \nu_{(1,1)}=2, \nu_{(i,j)}=0 \text{ for all other pairs of } (i,j)\}$.}, we see that the number of different frequency events $\mathcal{E}_{\{\nu_{(i,j)}\}}$ will not exceed $N^{n+1}$ and thus will be finite. Furthermore, it is easy to observe that two different frequency events must be disjoint. Thus we label all possible frequency events by the index $\lambda\in \Lambda$ where $\Lambda$ is a finite set, and we denote each frequency event by $\mathcal{E}_{\{\nu_{(i,j)}\}}^\lambda$. We claim that in the limit we have

\begin{theorem}[Asymptotic of the posterier probability]\label{Thm:AsymptoticPosterior}
For any $\varepsilon>0$ small enough, there exist some $M\geq 1$ and some $n_0=n_0(\varepsilon,M) \in \mathbb{N}$ 
such that for any $n\geq n_0$, there exists a family of frequency
events $\mathcal{E}_{\{\nu_{(i,j)}\}}^\lambda, \lambda\in \widetilde{\Lambda}\subseteq \Lambda$ such that
\begin{equation}\label{Thm:AsymptoticPosterior:Eq:FrequencyEventFullProb}
\mathbf{P}\left(\bigcup\limits_{\lambda\in \widetilde{\Lambda}} \mathcal{E}_{\{\nu_{(i,j)}\}}^\lambda\right)\geq 1-\dfrac{\varepsilon}{M} \ ,
\end{equation}
and for each frequency event $\mathcal{E}_{\{\nu_{(i,j)}\}}^\lambda$, $\lambda\in \widetilde{\Lambda}$, 
the posterior probability of $X_1$ conditioned on $\mathcal{E}_{\{\nu_{(i,j)}\}}^\lambda$ 
is close to the unconditioned probability of $X_1$, i.e. for any $1\leq i, j\leq N$ we have
\begin{equation}\label{Thm:AsymptoticPosterior:Eq:PosteriorEqualsPrior}
\left|\mathbf{P}\big(X_1=(i,j) \big|\mathcal{E}_{\{\nu_{(i,j)}\}}^\lambda\big) - 
\dfrac{\mathbf{1}^{i, \checkmark}_1 \pi^0_i}{\sum_{k_1=1}^N \mathbf{1}^{k_1, \checkmark}_1 \pi^0_{k_1}}\cdot p_{ij}\right|<\varepsilon \ .
\end{equation}
\end{theorem}

\begin{proof}
By using the conditional probability formula (\ref{Thm:Prop:MarkovChainPosteriorEqualsPrior:Eq:ConditionalProbabilityFormula}) we have 
\begin{equation}\label{Thm:AsymptoticPosterior:Eq:ConditionalProbabilityEstimate}
\begin{array}{ll}
& \left|\mathbf{P}\big(X_1=(i,j) \big|\mathcal{E}_{\{\nu_{(i,j)}\}}^\lambda\big) - 
\dfrac{\mathbf{1}^{i, \checkmark}_1 \pi^0_i}{\sum_{k_1=1}^N \mathbf{1}^{k_1, \checkmark}_1 \pi^0_{k_1}}\cdot p_{ij}\right|
\smallskip
\\
= & \left|\mathbf{P}(Y_1=i|\mathcal{E}_{\{\nu_{(i,j)}\}}^\lambda)
\cdot \mathbf{P}\big(X_1=(i,j) \big|\mathcal{E}_{\{\nu_{(i,j)}\}}^\lambda, Y_1=i\big)-\dfrac{\mathbf{1}^{i, \checkmark}_1 \pi^0_i}{\sum_{k_1=1}^N \mathbf{1}^{k_1, \checkmark}_1 \pi^0_{k_1}}\cdot p_{ij}\right|
\\
\stackrel{\rm{(a)}}{=} & \left|\mathbf{P}(Y_1=i|\mathcal{E}_{\{\nu_{(i,j)}\}}^\lambda)
\cdot \mathbf{P}\big(X_1=(i,j) \big|\mathcal{E}_{\{\nu_{(i,j)}\}}^\lambda, Y_1=i\big)-\mathbf{P}(Y_1=i|\mathcal{E}_{\{\nu_{(i,j)}\}}^\lambda)\cdot p_{ij}\right|
\\
= & \mathbf{P}(Y_1=i|\mathcal{E}_{\{\nu_{(i,j)}\}}^\lambda)\cdot \left|\mathbf{P}\big(X_1=(i,j) \big|\mathcal{E}_{\{\nu_{(i,j)}\}}^\lambda, Y_1=i\big)-p_{ij}\right|
\\
\stackrel{\rm{(b)}}{\leq} & \left|\mathbf{P}\big(X_1=(i,j) \big|\mathcal{E}_{\{\nu_{(i,j)}\}}^\lambda, Y_1=i\big)-p_{ij}\right| \ ,
\end{array}
\end{equation}
where in (a) we have used (\ref{Prop:MarkovChainPosteriorEqualsPrior:Y_1:Eq}) in Proposition \ref{Prop:MarkovChainPosteriorEqualsPrior:Y_1} and in (b) we have used the fact that $\mathbf{P}(Y_1=i|\mathcal{E}_{\{\nu_{(i,j)}\}}^\lambda)\leq 1$. 

Given $\varepsilon>0$, we will prove in Lemma \ref{Lm:AsymptoticMarkovChainPosteriorEqualsPrior} that there exists some $M_1\geq 1$ and some
$n_0^{(1)}=n_0^{(1)}(\varepsilon, M_1)\in \mathbb{N}$, so that for any $n\geq n_0^{(1)}$ and any frequency event $\mathcal{E}_{\{\nu_{(i,j)}\}}^\lambda$, $\lambda\in \widetilde{\Lambda}$ chosen as in Lemma \ref{Lm:AsymptoticFrequencyEvents} with $\mu=\dfrac{\varepsilon}{M_1}$, we have
\begin{equation}\label{Thm:AsymptoticPosterior:Eq:EpsilonEstimate(1):Frequency}
\left|\mathbf{P}\big(X_1=(i,j) \big|\mathcal{E}_{\{\nu_{(i,j)}\}}^\lambda, Y_1=i\big)-\dfrac{\nu_{(i,j)}}{\sum_{k=1}^N \nu_{(i, k)}}\right|< \dfrac{\varepsilon}{2} \ .
\end{equation}
 
Here we identify each frequency event $\mathcal{E}_{\{\nu_{(i,j)}\}}^\lambda$, $\lambda\in \widetilde{\Lambda}$ as such that each of the ratios $\dfrac{\nu_{(i,j)}}{n}$ is close to $\pi_i p_{ij}$, where $\{\pi_i\}_{i=1}^N$ is the invariant measure of $\{Y_\ell\}_{\ell\geq 1}$ introduced in Assumption \ref{Assmption:PositiveTransition} (see Lemma \ref{Lm:AsymptoticFrequencyEvents} for details).

Notice that by the simple ergodic theorem for Markov chains (see \cite[Section 1.10]{NorrisMarkovChain}), there exists some $n_0^{(2)}=n_0^{(2)}(\varepsilon)\in \mathbb{N}$ such that for any $n\geq n_0^{(2)}$ we have
\begin{equation}\label{Thm:AsymptoticPosterior:Eq:EpsilonEstimate(2):Ergodic}
\left|\dfrac{\nu_{(i,j)}}{\sum_{k=1}^N \nu_{(i, k)}}-p_{ij}\right|< \dfrac{\varepsilon}{2} \ ,
\end{equation}
for sufficiently small $\varepsilon>0$. 

Using (\ref{Thm:AsymptoticPosterior:Eq:ConditionalProbabilityEstimate}), (\ref{Thm:AsymptoticPosterior:Eq:EpsilonEstimate(1):Frequency}) and (\ref{Thm:AsymptoticPosterior:Eq:EpsilonEstimate(2):Ergodic}) we know that if we pick $M=\max(M_1, M_2)$ and $n_0=\max(n_0^{(1)}, n_0^{(2)})$, then we can form the set of frequency events $\mathcal{E}_{\{\nu_{(i,j)}\}}^\lambda$, $\lambda\in \widetilde{\Lambda}$ satisfying (\ref{Thm:AsymptoticPosterior:Eq:FrequencyEventFullProb}) and for any $n\geq n_0$, any such frequency event $\mathcal{E}_{\{\nu_{(i,j)}\}}^\lambda$, we must have
$$\begin{array}{ll}
& \left|\mathbf{P}\big(X_1=(i,j) \big|\mathcal{E}_{\{\nu_{(i,j)}\}}^\lambda\big) - 
\dfrac{\mathbf{1}^{i, \checkmark}_1 \pi^0_i}{\sum_{k_1=1}^N \mathbf{1}^{k_1, \checkmark}_1 \pi^0_{k_1}}\cdot p_{ij}\right|
\smallskip
\\
\stackrel{\rm{(a)}}{\leq} & \left|\mathbf{P}\big(X_1=(i,j) \big|\mathcal{E}_{\{\nu_{(i,j)}\}}^\lambda, Y_1=i\big)-p_{ij}\right| 
\smallskip
\\ 
\stackrel{\rm{(b)}}{\leq} & \left|\mathbf{P}\big(X_1=(i,j) \big|\mathcal{E}_{\{\nu_{(i,j)}\}}^\lambda, Y_1=i\big)-\dfrac{\nu_{(i,j)}}{\sum_{k=1}^N \nu_{(i, k)}}\right| + \left|\dfrac{\nu_{(i,j)}}{\sum_{k=1}^N \nu_{(i, k)}}-p_{ij}\right|
\\
\stackrel{\rm{(c)}}{<} & \dfrac{\varepsilon}{2}+\dfrac{\varepsilon}{2} = \varepsilon \ ,
\end{array}$$
where in (a) we have used (\ref{Thm:AsymptoticPosterior:Eq:ConditionalProbabilityEstimate}), in (b) we have used the triangle inequality, and in (c) we have used (\ref{Thm:AsymptoticPosterior:Eq:EpsilonEstimate(1):Frequency}) and (\ref{Thm:AsymptoticPosterior:Eq:EpsilonEstimate(2):Ergodic}). 
This implies (\ref{Thm:AsymptoticPosterior:Eq:PosteriorEqualsPrior}). 
\end{proof}

\begin{remark}[$Y_1$ cannot be any state]\label{Remark:FirstPositionCannotBeAnyStateInfiniteSample}
\rm
It is not true that when the number of samples $n$ is large we must have $\mathbf{1}_1^{i, \checkmark}=1$ for any $i\in \{1,2,...,N\}$. That is to say, not all states are admissible at the first position. This is simply because of (\ref{Eq:StringOfChainTypeFrequencyBalance}), since when the frequency event $\mathcal{E}^\lambda_{\{\nu_{(i,j)}\}}$ is given, the initial state $i_1$ and the final state $i_{n+1}$ must satisfy (\ref{Eq:StringOfChainTypeFrequencyBalance}). Because of this reason, the term $\dfrac{\mathbf{1}^{i, \checkmark}_1 \pi^0_i}{\sum_{k_1=1}^N \mathbf{1}^{k_1, \checkmark}_1 \pi^0_{k_1}}\cdot p_{ij}$ in (\ref{Thm:AsymptoticPosterior:Eq:PosteriorEqualsPrior}) cannot be replaced by $\pi^0_i p_{ij}$.
\end{remark}

\begin{lemma}[Asymptotic of Proposition \ref{Prop:MarkovChainPosteriorEqualsPrior}]\label{Lm:AsymptoticMarkovChainPosteriorEqualsPrior}
For any $\varepsilon>0$ small enough, there exists some $M\geq 1$ and some $n_0=n_0(\varepsilon, M)\in \mathbb{N}$, 
so that for any $n\geq n_0$ and any of the frequency events $\mathcal{E}^{\lambda}_{\{\nu_{(i,j)}\}}$ chosen from $\lambda\in \widetilde{\Lambda}$ as in \emph{Lemma \ref{Lm:AsymptoticFrequencyEvents}} with $\mu=\dfrac{\varepsilon}{M}$, 
for any $1\leq i, j \leq N$ we have 
\begin{equation}\label{Lm:AsymptoticMarkovChainPosteriorEqualsPrior:Eq}
\left|\mathbf{P}\big(X_1=(i,j) \big|\mathcal{E}_{\{\nu_{(i,j)}\}}^\lambda, Y_1=i\big)
-\dfrac{\nu_{(i,j)}}{\sum_{k=1}^N \nu_{(i, k)}}\right|<\dfrac{\varepsilon}{2} \ .
\end{equation}
\end{lemma}

\begin{proof}
Suppose we have been given one of any frequency events 
$\mathcal{E}_{\{\nu_{(i,j)}\}}^\lambda$ with $\lambda\in \Lambda$ and also $Y_1=i$. Under this assumption, 
recall that in our proof of Proposition \ref{Prop:MarkovChainPosteriorEqualsPrior}, we have introduced the set $\mathcal{X}^{\text{distinguished}}$ in (\ref{Prop:MarkovChainPosteriorEqualsPrior:Eq:DistinguishedOccurances}) and all possible length-$n$ ordered sequences $\mathscr{S}$ consisting of distinguished elements of the form $((i,j),a)\in \mathcal{X}^{\text{distinguished}}$, such that the element-wise projection $\mathfrak{Y}$ with $((i,j),a)\stackrel{\mathfrak{Y}}{\rightarrow} (i,j)$ applied to each of the above sequence $\mathscr{S}$ results in a string of chain type $\mathfrak{Y}(\mathscr{S})$ with head $i$. 
By Fact 1' in the proof of Proposition \ref{Prop:MarkovChainPosteriorEqualsPrior}, we know that 
$$\begin{array}{ll}
& \mathscr{P}\left(\text{all possible }\mathscr{S} \text{ whose 1-st element is } ((i,j),a)\right)
\\
= & 
(\text{Number of all possible } \mathscr{S} \text{ whose 1-st element is } ((i,j),a))\cdot p
\end{array}$$
where $p\in (0,1)$ is independent of $((i,j), a)$. Since we have 
$$\sum\limits_{((i,j),a)\in \mathcal{X}^{\text{distinguished}}, i \text{ fixed}}
\mathscr{P}\left(\text{all possible }\mathscr{S} \text{ whose 1-st element is } ((i,j),a)\right)=1 \ ,$$
and by Fact 4' stated in the proof of Proposition \ref{Prop:MarkovChainPosteriorEqualsPrior}, indicating that $$(\text{Number of all possible } \mathscr{S} \text{ whose 1-st element is } ((i,k),a))$$ 
is independent of $a$, we obtain that for any choices of $1\leq a\leq \nu_{(i,k)}$, 
$$1=\sum\limits_{k=1}^N 
\nu_{(i,k)}\cdot (\text{Number of all possible } \mathscr{S} \text{ whose 1-st element is } ((i,k),a))\cdot p \ ,
$$
which gives
\begin{equation}\label{Lm:AsymptoticMarkovChainPosteriorEqualsPrior:Eq:Determine-p}
p=\dfrac{1}{\sum\limits_{k=1}^N 
\nu_{(i,k)}\cdot (\text{Number of all possible } \mathscr{S} \text{ whose 1-st element is } ((i,k),a))} \ .
\end{equation}

We recall the formula (\ref{Prop:MarkovChainPosteriorEqualsPrior:Eq:PosterialMarginalDecomposition:Step:Counting}) in the proof of Proposition \ref{Prop:MarkovChainPosteriorEqualsPrior} and we combine it with the above (\ref{Lm:AsymptoticMarkovChainPosteriorEqualsPrior:Eq:Determine-p}) to obtain
\begin{equation}\label{Lm:AsymptoticMarkovChainPosteriorEqualsPrior:Eq:Step:Counting}
\begin{array}{ll}
& \mathbf{P}(X_1=(i,j)|\mathcal{E}_{\{\nu_{(i,j)}\}}^\lambda , Y_1=i)
\\
= &  \mathbf{1}^{j, \checkmark}_2
\cdot \sum\limits_{a'=1}^{\nu_{(i,j)}} \mathscr{P}\left(\text{all possible }\mathscr{S} \text{ whose 1-st element is } ((i,j),a')\right)
\\
= & \mathbf{1}^{j, \checkmark}_2
\cdot \sum\limits_{a'=1}^{\nu_{(i,j)}} \dfrac{(\text{Number of all possible } \mathscr{S} \text{ whose 1-st element is } ((i,j),a'))}{\sum\limits_{k=1}^N 
\nu_{(i,k)}\cdot (\text{Number of all possible } \mathscr{S} \text{ whose 1-st element is } ((i,k),a))}
\\
= & \mathbf{1}^{j, \checkmark}_2
\cdot \sum\limits_{a'=1}^{\nu_{(i,j)}} \dfrac{1}{\sum\limits_{k=1}^N 
\nu_{(i,k)}\cdot \gamma(k,j)} \ ,
\end{array}  
\end{equation}
where we denote 
$$\gamma(k,j)=\frac{\text{Number of all possible } \mathscr{S} \text{ whose 1-st element is } ((i,k),a)}
{\text{Number of all possible } \mathscr{S} \text{ whose 1-st element is } ((i,j),a')} \ .$$
Notice that $\gamma(k,j)$ is independent of $a$ and $a'$. The finite existence of $\gamma(k,j)$ is guarentted by part (a) of Lemma \ref{Lm:AsymptoticConditionalSymmetry}.

Given any $\varepsilon>0$, by part (b) of Lemma \ref{Lm:AsymptoticConditionalSymmetry} we know that there exists some $M_1\geq 1$ and some $n_0^{(1)}=n_0^{(1)}(\varepsilon, M_1)\in \mathbb{N}$ such that for any $n\geq n_0^{(1)}$, 
for any of our given frequency events $\mathcal{E}_{\{\nu_{(i,j)}\}}^\lambda$, $\lambda\in \widetilde{\Lambda}$ chosen as in Lemma 1 with $\mu=\dfrac{\varepsilon}{M_1}$ and any $1\leq k,j\leq N$, that
$$|\gamma(k,j)-1|< \dfrac{\varepsilon}{4} \ .$$ 
This combined with (\ref{Lm:AsymptoticMarkovChainPosteriorEqualsPrior:Eq:Step:Counting}) enable us to estimate 
\begin{equation}\label{Lm:AsymptoticMarkovChainPosteriorEqualsPrior:Eq:FirstEstimate}
\begin{array}{ll}
& \left|\mathbf{P}(X_1=(i,j)|\mathcal{E}_{\{\nu_{(i,j)}\}}^\lambda , Y_1=i)-\mathbf{1}^{j, \checkmark}_2 
\dfrac{\nu_{(i,j)}}{\sum_{k=1}^N \nu_{(i,k)}}\right|
\\
\leq & \sum\limits_{a=1}^{\nu_{(i,j)}}\left|\dfrac{1}{\sum_{k=1}^N 
\nu_{(i,k)}\cdot \gamma(k,j)}-\dfrac{1}{\sum_{k=1}^N \nu_{(i,k)}}\right|
= \sum\limits_{a=1}^{\nu_{(i,j)}}
\dfrac{\sum_{k=1}^N \nu_{(i,k)}\left|\gamma(k,j)-1\right|}
{\left(\sum_{k=1}^N \nu_{(i,k)}\cdot \gamma(k,j)\right)
\cdot \left(\sum_{k=1}^N \nu_{(i,k)}\right)}
\\
< & \dfrac{\varepsilon}{4}\cdot \sum\limits_{a=1}^{\nu_{(i,j)}}
\dfrac{\sum_{k=1}^N \nu_{(i,k)}}
{\left(\sum_{k=1}^N \nu_{(i,k)}\cdot \gamma(k,j)\right)
\cdot \left(\sum_{k=1}^N \nu_{(i,k)}\right)}
= \dfrac{\varepsilon}{4} \cdot 
\dfrac{\nu_{(i,j)}}
{\sum_{k=1}^N \nu_{(i,k)}\cdot \gamma(k,j)}
\\
< & \dfrac{\frac{\varepsilon}{4}}{1-\frac{\varepsilon}{4}}\cdot \dfrac{\nu_{(i,j)}}
{\sum_{k=1}^N \nu_{(i,k)}}
\leq \dfrac{\frac{\varepsilon}{4}}{1-\frac{\varepsilon}{4}} \ .
\end{array}
\end{equation}
By Lemma \ref{Lm:AsymptoticIndicator} we know that there exists some $M_2\geq 1$ and some $n_0^{(2)}=n_0^{(2)}(\varepsilon, M_2)\in \mathbb{N}$ and when $n\geq n_0^{(2)}$, for any of our frequency events $\mathcal{E}_{\{\nu_{(i,j)}\}}^\lambda$, $\lambda\in \widetilde{\Lambda}$ chosen as in Lemma \ref{Lm:AsymptoticFrequencyEvents} with $\mu=\dfrac{\varepsilon}{M_2}$, we have 
$$|\mathbf{1}^{j, \checkmark}_2-1|=0 \ .$$ 

We pick $M=\max(M_1, M_2)$ and $n_0=\max(n_0^{(1)}, n_0^{(2)})$. By using (\ref{Lm:AsymptoticMarkovChainPosteriorEqualsPrior:Eq:FirstEstimate}), we can form the set of frequency events $\mathcal{E}_{\{\nu_{(i,j)}\}}^\lambda$, $\lambda\in \widetilde{\Lambda}$ satisfying Lemma \ref{Lm:AsymptoticFrequencyEvents} with $\mu=\dfrac{\varepsilon}{M}$ and for any $n\geq n_0$, any such frequency event $\mathcal{E}_{\{\nu_{(i,j)}\}}^\lambda$, we must have
\begin{equation*}\label{Lm:AsymptoticMarkovChainPosteriorEqualsPrior:Eq:SecondEstimate}
\begin{array}{ll}
& \left|\mathbf{P}(X_1=(i,j)|\mathcal{E}_{\{\nu_{(i,j)}\}}^\lambda , Y_1=i)-\dfrac{\nu_{(i,j)}}{\sum_{k=1}^N \nu_{(i,k)}}\right|
\smallskip
\\
\leq & \left|\mathbf{P}(X_1=(i,j)|\mathcal{E}_{\{\nu_{(i,j)}\}}^\lambda , Y_1=i)-\mathbf{1}^{j, \checkmark}_2 
\dfrac{\nu_{(i,j)}}{\sum_{k=1}^N \nu_{(i,k)}}\right|+
\left|\mathbf{1}^{j, \checkmark}_2-1\right|\cdot
\dfrac{\nu_{(i,j)}}{\sum_{k=1}^N \nu_{(i,k)}}
\\
< & \dfrac{\frac{\varepsilon}{4}}{1-\frac{\varepsilon}{4}} \leq \dfrac{\varepsilon}{2} \ ,
\end{array}
\end{equation*}
when $0<\varepsilon\leq 2$. This proves (\ref{Lm:AsymptoticMarkovChainPosteriorEqualsPrior:Eq}).
\end{proof}

\begin{remark}\rm\label{Remark:AlternateProof:Prop:MarkovChainPosteriorEqualsPrior}
As we have explained in Remark \ref{Remark:AlternateProof:Prop:IIDPosteriorEqualsPrior}, an alternate and seemingly simpler proof of the above Lemma \ref{Lm:AsymptoticMarkovChainPosteriorEqualsPrior} can be obtained by directly using Proposition \ref{Prop:MarkovChainPosteriorEqualsPrior} and Corollary \ref{Corollary:NumberOfChainsStringType}. The argument is parallel to the one used in the proof of Lemma \ref{Lm:AsymptoticConditionalSymmetry} below. We omit details here. However, the way of proof we adopt here reveals more of the underlying symmetric structure of the problem (i.e. ``conditional symmetry" at different levels, see Remark \ref{Remark:MarkovChainAsymptoticConditionalSymmetryAtLevelOfObservations}). We expect that such arguments based on symmetry are more fundamental and should be extended to more general classes of processes. 
\end{remark}

\begin{lemma}[Asymptotic of frequency events]\label{Lm:AsymptoticFrequencyEvents}
For any $\mu>0$ there exists some $n_0=n_0(\mu)\in \mathbb{N}$ such that for any $n\geq n_0$, 
there exists a family of frequency
events $\mathcal{E}_{\{\nu_{(i,j)}\}}^\lambda, \lambda\in \widetilde{\Lambda}\subseteq \Lambda$ with 
\begin{equation}\label{Lm:AsymptoticFrequencyEvents:Eq:FrequencyEventFullProb}
\mathbf{P}\left(\bigcup\limits_{\lambda\in \widetilde{\Lambda}} \mathcal{E}_{\{\nu_{(i,j)}\}}^\lambda\right)\geq 1-\mu \ ,
\end{equation}
and for each frequency event $\mathcal{E}_{\{\nu_{(i,j)}\}}^\lambda$, $\lambda\in \widetilde{\Lambda}$, its corresponding frequencies $\nu_{(i,j)}$ satisfy that for any $1\leq i,j\leq N$, 
\begin{equation}\label{Lm:AsymptoticFrequencyEvents:Eq:FrequencyIsProbability}
\left|\dfrac{\nu_{(i,j)}}{n}-\pi_i p_{ij}\right|<\mu \ ,
\end{equation}
where $\pi_i, i=1,2,...,N$ is the invariant measure of the Markov chain $\{Y_\ell\}_{\ell\geq 1}$ and $p_{ij}$ are the transition probabilities.
\end{lemma}

\begin{proof}
By our Assumption \ref{Assmption:PositiveTransition}, 
from the weak Law of Large Numbers for ergodic Markov chain (see \cite[Section 5.4]{KoralovTextbook}) we know that
for any $\mu>0$, there exists some $n_0=n_0(\mu)\in \mathbb{N}$ and for all $n\geq n_0$ we have
$$\mathbf{P}\left(\left|\dfrac{\nu_{(i,j)}}{n} - \pi_i p_{ij}\right|\geq \mu\right)\leq \mu \ ,$$
which implies the statement of the Lemma.
\end{proof}

\begin{remark}[Desired frequency events cannot be of full probability]\rm\label{Remark:Lm:AsymptoticFrequencyEvents:CannotUseStrongLLN}
It is not appropriate to conclude here that $\mathbf{P}\left(\bigcup\limits_{\lambda\in \widetilde{\Lambda}} \mathcal{E}_{\{\nu_{(i,j)}\}}^\lambda\right)=1$ by using the strong Law of Large Numbers for ergodic Markov chain (see
\cite[Section 1.10]{NorrisMarkovChain}), since in that case the threshold $n_0$ may depend on the element $\omega$ in the probability space $\Omega$, i.e. the assertion that
$$\mathbf{P}\left(\lim\limits_{n\rightarrow \infty}\dfrac{\nu_{(i,j)}}{n}=\pi_i p_{ij}\right)=1$$
implies that for some $\widehat{\Omega}\subseteq \Omega$ with $\mathbf{P}(\widehat{\Omega})=1$, for any $\mu>0$ and any $\omega\in \widehat{\Omega}$, there exists some $n_0\in \mathbb{N}$ that may depend on $\omega$, such that $\displaystyle{\left|\dfrac{\nu_{(i,j)}}{n}-\pi_i p_{ij}\right|<\mu}$.
\end{remark}

Recall that the indicator function $\mathbf{1}_\ell^{i, \checkmark}$ is defined as in Definition \ref{Def:ConditionalAdmissibleStates}, which indicates that state $i$ is conditionally admissble at $Y_\ell$ given the observed frequencies $\mathcal{E}_{\{\nu_{(i,j)}\}}$. As the number of observations $n$ tends to infinity, we have 

\begin{lemma}[All states are asymptotically conditionally admissible at $Y_2$]\label{Lm:AsymptoticIndicator}
For any $j\in \{1,2,...,N\}$ and any $1\leq \ell \leq n$, for any $\varepsilon>0$ and any $M\geq 1$,
there exist some $n_0\in \mathbb{N}$ such that 
when $n\geq n_0$, for any frequency event $\mathcal{E}_{\{\nu_{(i,j)}\}}^\lambda$ chosen from $\lambda\in \widetilde{\Lambda}$ as in \emph{Lemma \ref{Lm:AsymptoticFrequencyEvents}} with $\mu=\dfrac{\varepsilon}{M}$ and admitting $Y_1=i$, we have 
\begin{equation}\label{Lm:AsymptoticIndicator:Eq}
\mathbf{1}_2^{j, \checkmark}= 1 \ .
\end{equation}
That is to say, in the $n\rightarrow\infty$ asymptotic, all states are conditionally admissible at $Y_2$.
\end{lemma}

\begin{proof}
Without loss of generality we shall suppress the upper-index $\lambda$ in $\mathcal{E}_{\{\nu_{(i,j)}\}}^\lambda$. Since $N_{iv}^{(n)}(\mathcal{E}_{\{\nu_{(i,j)}\}})\geq 1$, by (\ref{Corollary:NumberOfChainsStringType:Eq:ModifiedNumberAdmissiblePathsFinalForm}) in the proof of Corollary \ref{Corollary:NumberOfChainsStringType} we know that
$$N_{jv}^{(n-1)}(\mathcal{E}_{\{\widetilde{\nu}_{(\widetilde{i},\widetilde{j})}\}})=\dfrac{\nu_{(i,j)}}
{\sum\limits_{j=1}^N \nu_{(i,j)}}\cdot 
N_{iv}^{(n)}(\mathcal{E}_{\{\nu_{(i,j)}\}})\cdot 
\dfrac{\widetilde{F}_{vj}^*}{F_{vi}^*}  \ ,$$
where we use the same notations in as Corollary \ref{Corollary:NumberOfChainsStringType} and we refer to the reader for more details there. We then see that in order to prove (\ref{Lm:AsymptoticIndicator:Eq}) it suffices to show that 
$N_{jv}^{(n-1)}(\mathcal{E}_{\{\widetilde{\nu}_{(\widetilde{i},\widetilde{j})}\}})\geq 1$. 
As $n$ is large, by Lemma \ref{Lm:AsymptoticFrequencyEvents} we know that for any $1\leq i,j\leq N$ we must have $\nu_{(i,j)}\geq 1$. By Lemma \ref{Lm:WhittleFormulaF*F*TildeAsymptoticallySame} we know that for any $1\leq j\leq N$ we have $\widetilde{F}_{vj}^*\rightarrow F^*_{vj}$ as $n\rightarrow\infty$. By (\ref{Lm:AsymptoticFrequencyEvents:Eq:FrequencyIsProbability}) we know that the assumption of Lemma \ref{Lm:WhittleFormulaF*CofactorSame} is satisfied, and thus we know that 
$\dfrac{F_{vj}^*}{F_{vi}^*}=1$. These facts imply that $N_{jv}^{(n-1)}(\mathcal{E}_{\{\widetilde{\nu}_{(\widetilde{i},\widetilde{j})}\}})\geq 1$ when $n$ is large, and thus (\ref{Lm:AsymptoticIndicator:Eq}) is proved. 
\end{proof}

Recall in Remark \ref{Remark:MarkovChainProofIllustratesConditionalSymmetry} we have explained that the Fact 4' in the proof of Proposition \ref{Prop:MarkovChainPosteriorEqualsPrior} is different from the ``conditional symmetry" as the Fact 4 of the proof of Theorem \ref{Thm:IIDPosteriorEqualsPrior}. Here we show that as $n\rightarrow \infty$, Fact 4' will asymptotically become ``conditionally symmetric" and thus it becomes a version of Fact 4 stated in the proof of Theorem \ref{Thm:IIDPosteriorEqualsPrior}. Using the same notations as in the proof of Proposition \ref{Prop:MarkovChainPosteriorEqualsPrior}, we have

\begin{lemma}[Asymptotic conditional symmetry]\label{Lm:AsymptoticConditionalSymmetry}
For any $\varepsilon>0$, there exist some $M\geq 1$ and some $n_0\in \mathbb{N}$ such that when $n\geq n_0$, for any 
of the frequency events $\mathcal{E}_{\{\nu_{(i,j)}\}}^\lambda$ chosen from $\lambda\in \widetilde{\Lambda}$ as in \emph{Lemma \ref{Lm:AsymptoticFrequencyEvents}} with $\mu=\dfrac{\varepsilon}{M}$ and admitting $Y_1=i\in \{1,2,...,N\}$, we have

\begin{itemize}
\item[(a)] For any $j\in \{1,2,...,N\}$ and any $1\leq a\leq \nu_{(i,j)}$, 
\emph{$$\text{Number of all possible } \mathscr{S} \text{ whose 1-st element is } ((i,j),a)\geq 1 \ ;$$}
\item[(b)] For any two elements \emph{$((i,j_1), a_1)$, $((i, j_2), a_2)\in \mathcal{X}^{\text{distinguished}}$} with $j_1\neq j_2$, $j_{1,2}\in \{1,2,...,N\}$ and $1\leq a_1\leq \nu_{(i, j_1)}, 1\leq a_2\leq \nu_{(i, j_2)}$, we have
\emph{\begin{equation}\label{Lm:AsymptoticConditionalSymmetry:Eq}
\left|\dfrac{\text{Number of all possible } \mathscr{S} \text{ whose 1-st element is } ((i,j_1),a_1)}
{\text{Number of all possible } \mathscr{S} \text{ whose 1-st element is } ((i,j_2),a_2)}- 1\right|< \dfrac{\varepsilon}{4} \ .
\end{equation}}
\end{itemize}
\end{lemma}

\begin{proof}
We make use of Fact 4' in the proof of Proposition \ref{Prop:MarkovChainPosteriorEqualsPrior}, so that for each element $((i,j),a)\in \mathcal{X}^{\text{distinguished}}$ with $i$ being fixed, 
$$\begin{array}{ll}
&\left(\text{Number of all possible } \mathscr{S} \text{ whose 1-st element is } ((i,j),a)\right)
\\
=&\dfrac{\nu_{(i_1, j_1)}! \, ... \, \nu_{(i_I, j_I)}!}{\nu_{(i,j)}}\cdot \mathbf{1}^{j,\checkmark}_2\cdot\text{\#}_1^{(i,j)}(\mathcal{E}^\lambda_{\{\nu_{(i,j)}\}}) \ .
\end{array}$$ 

Since $Y_1=i\in \{1,2,...,N\}$ is admitted by the frequency event $\mathcal{E}^\lambda_{\{\nu_{(i,j)}\}}$, we know that ${\#}_1^{(i,j)}(\mathcal{E}^\lambda_{\{\nu_{(i,j)}\}})\geq 1$. By Lemma \ref{Lm:AsymptoticIndicator} we know that when $n$ is large, $\mathbf{1}_2^{j,\checkmark}=1$. Thus part (a) is proved. 

For part (b), by (\ref{Corollary:NumberOfChainsStringType:Eq}) in Corollary \ref{Corollary:NumberOfChainsStringType}, we know that
$$\text{\#}^{(i,j)}_1(\mathcal{E}_{\{\nu_{(i,j)}\}})=\dfrac{\nu_{(i,j)}}{\sum\limits_{j=1}^N \nu_{(i,j)}}\cdot 
N_{iv}^{(n)}(\mathcal{E}_{\{\nu_{(i,j)}\}})\cdot 
\dfrac{\widetilde{F}_{vj}^*}{F_{vi}^*} \ .$$
Therefore
$$\begin{array}{ll}
& \left(\text{Number of all possible } \mathscr{S} \text{ whose 1-st element is } ((i,j),a)\right)
\\
= & \dfrac{\nu_{(i_1, j_1)}! \, ... \, \nu_{(i_I, j_I)}!}{\nu_{(i,j)}}\cdot \mathbf{1}^{j,\checkmark}_2\cdot
\dfrac{\nu_{(i,j)}}{\sum\limits_{j=1}^N \nu_{(i,j)}}\cdot 
N_{iv}^{(n)}(\mathcal{E}_{\{\nu_{(i,j)}\}})\cdot 
\dfrac{\widetilde{F}_{vj}^*}{F_{vi}^*}
\\
= & \dfrac{\nu_{(i_1, j_1)}! \, ... \, \nu_{(i_I, j_I)}!}{{\sum\limits_{j=1}^N \nu_{(i,j)}}} \cdot 
\mathbf{1}^{j,\checkmark}_2 \cdot N_{iv}^{(n)}(\mathcal{E}_{\{\nu_{(i,j)}\}})\cdot 
\dfrac{\widetilde{F}_{vj}^*}{F_{vi}^*} \ .
\end{array}$$ 
Here the matrices $F^*=(\nu_{ij}^*)_{1\leq i, j\leq N}$ and $\widetilde{F}^*=(\widetilde{\nu}^*_{\widetilde{i} \, \widetilde{j}})_{1\leq \widetilde{i}, \widetilde{j}\leq N}$ are defined in (\ref{Eq:WhittleFormulaF*}) and (\ref{Eq:WhittleFormulaF*Tilde}); and $\widetilde{F}_{vj}^*$, $F_{vi}^*$ are the $(v,j)$-th and $(v,i)$-th cofactor of these two matrices, respectively. 

Thus we have
\begin{equation}\label{Lm:AsymptoticConditionalSymmetry:Eq:Ratio}
\begin{array}{ll}
& \dfrac{\text{Number of all possible } \mathscr{S} \text{ whose 1-st element is } ((i,j_1),a_1)}
{\text{Number of all possible } \mathscr{S} \text{ whose 1-st element is } ((i,j_2),a_2)}
\\
= & \dfrac{\mathbf{1}^{j_1,\checkmark}_2\cdot  N_{iv}^{(n)}(\mathcal{E}_{\{\nu_{(i,j)}\}})\cdot 
\dfrac{\widetilde{F}_{vj_1}^*}{F_{vi}^*}}
{\mathbf{1}^{j_2,\checkmark}_2\cdot 
N_{iv}^{(n)}(\mathcal{E}_{\{\nu_{(i,j)}\}})\cdot 
\dfrac{\widetilde{F}_{vj_2}^*}{F_{vi}^*}} \ .
\end{array}
\end{equation}
By Lemma \ref{Lm:WhittleFormulaF*F*TildeAsymptoticallySame} we know that for any $1\leq j\leq N$ we have $\widetilde{F}_{vj}^*\rightarrow F^*_{vj}$ as $n\rightarrow\infty$. By (\ref{Lm:AsymptoticFrequencyEvents:Eq:FrequencyIsProbability}) we know that the assumption of Lemma \ref{Lm:WhittleFormulaF*CofactorSame} is satisfied, and thus we know that 
$\dfrac{F_{vj}^*}{F_{vi}^*}=1$. These facts together with (\ref{Lm:AsymptoticConditionalSymmetry:Eq:Ratio}) and Lemma \ref{Lm:AsymptoticIndicator} imply (\ref{Lm:AsymptoticConditionalSymmetry:Eq}). So we have also proved part (b).
\end{proof}

\begin{remark}[Asymptotic Conditional Symmetry at the level of observations of one-step transitions]\rm\label{Remark:MarkovChainAsymptoticConditionalSymmetryAtLevelOfObservations}
From Lemma \ref{Lm:AsymptoticConditionalSymmetry} we see that, as the number of observations $n$ tends to infinity, a higher level of symmetry is manifested at the fact that the numbers of admissible trajectories starting from different initial one-step transitions tend to be evenly distributed. From here, the ``conditional symmetry" at the level of sample path trajectories as shown in Fact 4' of the proof of Proposition \ref{Prop:MarkovChainPosteriorEqualsPrior} is reduced to the ``conditional symmetry" at the level of observations of one-step transitions. The latter level of symmetry is essentially the same as Fact 4 in the proof of Theorem \ref{Thm:IIDPosteriorEqualsPrior} for the i.i.d. case. By this reason, when $n\rightarrow \infty$, the ergodic finite Markov chain case should have the same kind of posterior distribution as the i.i.d case, with the only difference of replacing the frequencies of state occurrences by the frequencies of state transitions. This is exactly what Lemma \ref{Lm:AsymptoticMarkovChainPosteriorEqualsPrior} indicates.  
\end{remark}

Recall that we have defined the matrices $F^*=(\nu_{ij}^*)_{1\leq i, j\leq N}$ and $\widetilde{F}^*=(\widetilde{\nu}^*_{\widetilde{i} \, \widetilde{j}})_{1\leq \widetilde{i}, \widetilde{j}\leq N}$ in (\ref{Eq:WhittleFormulaF*}) and (\ref{Eq:WhittleFormulaF*Tilde}), respectively. The following lemma shows that under Assumption \ref{Assmption:PositiveTransition}, these two matrices have asymptotically the same elements.

\begin{lemma}[Asymptotic of the matrices $F^*$ and $\widetilde{F}^*$]\label{Lm:WhittleFormulaF*F*TildeAsymptoticallySame}
For the given $1\leq i,j\leq N$ that are used to define $\widetilde{F}^*=(\widetilde{\nu}^*_{\widetilde{i} \, \widetilde{j}})_{1\leq \widetilde{i}, \widetilde{j}\leq N}$ in \emph{(\ref{Eq:WhittleFormulaF*Tilde})}, assume that $\nu_{(i,j)}\geq 1$ in $\mathcal{E}_{\{\nu_{(i,j)}\}}$ and the \emph{Assumption \ref{Assmption:PositiveTransition}} holds. Then for any $\varepsilon>0$ there exists some $n_0\geq 1$ such that when $n\geq n_0$ and for any pair $(k,l), 1\leq k, l\leq N$ we have 
\begin{equation}\label{Lm:WhittleFormulaF*F*TildeAsymptoticallySame:Eq}
\left|\widetilde{\nu}^*_{kl}-\nu^*_{kl}\right|<\varepsilon \ .
\end{equation}
\end{lemma}

\begin{proof}
According to (\ref{Eq:WhittleFormulaF*Tilde}), if $k\neq i$, then we actually have $\widetilde{\nu}^*_{kl}=\nu^*_{kl}$, so (\ref{Lm:WhittleFormulaF*F*TildeAsymptoticallySame:Eq}) is automatically true. Now suppose $k=i$, then since $\nu_{(i,j)}\geq 1$, we must have $\sum\limits_{l=1}^N \nu_{(i, l)}>0$. Thus according to (\ref{Eq:WhittleFormulaF*}), for each $1\leq l\leq N$ we have 
\begin{equation}\label{Lm:WhittleFormulaF*F*TildeAsymptoticallySame:Eq:FormulaForNu}
\nu^*_{il}=\mathbf{1}_{\{i=l\}}-\dfrac{\nu_{(i,l)}}{\sum\limits_{l=1}^N \nu_{(i,l)}} \ .
\end{equation}
Without loss of generality we can assume that $n$ is large, so $\sum\limits_{l=1}^{N} \nu_{(i,l)}=1 \ , \ \nu_{(i,j)}= 1$ will not happen, because otherwise we must have $n=1$. Thus we know from (\ref{Eq:WhittleFormulaF*Tilde}) that 
\begin{equation}\label{Lm:WhittleFormulaF*F*TildeAsymptoticallySame:Eq:FormulaForNu*}
\widetilde{\nu}^*_{il}=\left\{
\begin{array}{ll}
\mathbf{1}_{\{i=j\}}-\dfrac{\nu_{(i,j)}-1}{\sum\limits_{l=1}^{N} \nu_{(i,l)}-1} \ , & 
\text{ if } l=j \ ,
\\
\mathbf{1}_{\{i=l\}}-\dfrac{\nu_{(i,l)}}{\sum\limits_{l=1}^{N} \nu_{(i,l)}-1} \ , &
\text{ if } l\neq j \ .
\end{array}
\right.
\end{equation}
From (\ref{Lm:WhittleFormulaF*F*TildeAsymptoticallySame:Eq:FormulaForNu*}) and (\ref{Lm:WhittleFormulaF*F*TildeAsymptoticallySame:Eq:FormulaForNu}) we can calculate that when $l=j$, we have
$$|\nu^*_{il}-\widetilde{\nu}^*_{il}|=\left|\dfrac{\nu_{(i,j)}}{\sum\limits_{l=1}^N \nu_{(i,l)}}
- \dfrac{\nu_{(i,j)}-1}{\sum\limits_{l=1}^{N} \nu_{(i,l)}-1}\right|
=\dfrac{\sum\limits_{l=1}^N \nu_{(i,l)}-\nu_{(i,j)}}{\left(\sum\limits_{l=1}^N \nu_{(i,l)}\right)\left(\sum\limits_{l=1}^{N} \nu_{(i,l)}-1\right)}< \dfrac{1}{\sum\limits_{l=1}^{N} \nu_{(i,l)}-1} \ ,$$
and when $l\neq j$ we have
$$|\nu^*_{il}-\widetilde{\nu}^*_{il}|=\left|\dfrac{\nu_{(i,l)}}{\sum\limits_{l=1}^N \nu_{(i,l)}}
- \dfrac{\nu_{(i,l)}}{\sum\limits_{l=1}^{N} \nu_{(i,l)}-1}\right|
=\dfrac{\nu_{(i,l)}}{\left(\sum\limits_{l=1}^N \nu_{(i,l)}\right)\left(\sum\limits_{l=1}^{N} \nu_{(i,l)}-1\right)}< \dfrac{1}{\sum\limits_{l=1}^{N} \nu_{(i,l)}-1} \ .$$
Thus in order that (\ref{Lm:WhittleFormulaF*F*TildeAsymptoticallySame:Eq}) holds, it suffices to have 
$$\sum\limits_{l=1}^{N} \nu_{(i,l)}-1>\dfrac{1}{\varepsilon} \ ,$$
which is a result of (\ref{Lm:AsymptoticFrequencyEvents:Eq:FrequencyIsProbability}) in Lemma \ref{Lm:AsymptoticFrequencyEvents}. 
\end{proof}

\section{Discussion, Remarks, and Generalization}
\label{Sec:discussions}

\subsection{Koopman-Damois exponential family of models and maximum entropy principle}

Recent studies on applying the probability theory of large deviations to nanothermodynamics 
\cite{hill-book,lu-qian-20,cyq-21,time-event-thermo-I} have shed considerable new light on the nature of Gibbs' equilibrium theory of statistical mechanics and statistical chemistry: The former introduces the notion of statistical ensemble for the probability of energy and system's volume, and the latter generalizes the method to counting the number of atoms and molecules. It is clear now that Gibbs' theory, in fact the entire theory of thermodynamics proper, is an approach that combines Bayesian statistical inference and limit theorems by formulating a {\em posterior probability} for a representative member in a large system that is conditioned on a limit law.  In fact, each specific limit law also implies a family of probabilistic models with parameters (conjugate variables) that can be determined by the empirical observation as sufficient statistics.  The last aspect of thermodynamics is precisely the {\em phenomenological thermodynamics that accounts for fluctuations} first proposed by L. Szilard in 1925 \cite{szilard} and further developed by B. Mandelbrot \cite{mandelbrot_1962}, in terms of the Koopman-Damois (KD), also known as exponential, family of models \cite{koopman_1936,jeffreys_1960}.  The sufficient statistics also has a deep logic relation to the Maximum Entropy Principle (MEP) which uniquely determines the posterior distribution based on idealized ``date'', as elucidated in \cite{vanCampenhout_1981}.

The KD family is a consequence of the large-deviation posterior inference, via exponential tilting \cite{barato_2015,touchette_2015}.  The present work further shows that such a factorized form is valid even for posterior conditioned on finite observations.  More specifically, one has the measure-theoretic formulation
\begin{subequations}
\label{eqn17}
\begin{eqnarray}
     \mathbf{P}\big( \rd x \big| \hat{\theta} \big)
      &=& \exp\big[\eta\big(\hat{\theta}\big)g(x)
        - \Lambda\big(\hat{\theta}\big) \big] \mu(\rd x)
\\
      &\asymp& \exp\big[I\big(g(x)\big)\big]\mu(\rd x),
\\
       \Lambda(\theta) &=& \sup_{g}
        \big\{\eta(\theta) g -I(g)\big\},
\\
       \hat{\theta}\big(g^*\big) &=& \arg\sup_{\theta} \big\{ \eta(\theta) g^* -\Lambda(\theta)\big\}.
\end{eqnarray}
\end{subequations}
w.r.t. a reference measure $\mu$.  In (\ref{eqn17}), $\eta$ is the {\em conjugate variable} to the random variable $g(x)$, its cumulant generating function $\Lambda(g)$ is the Legendre-Fenchel transform of the LDRF $I(g)$, and $g^*$ is the empirical arithmetic mean value of observing $g$.  In classical thermodynamics, 
$I(g)$ and $\Lambda(\hat{\theta})$ are identified
with the Gibbs entropy and the Helmholtz free energy, respectively, when $\eta$ is temperature and $g$ is internal energy.  Gibbs' statistical mechanics approach with explicit consideration of atoms and molecules and Szilard-Mandelbrot's statistical inference approach to phenomenological thermodynamics including fluctuations, thus, are now unified under the Probability Theory. 

Indeed, contraction principle in the theory of large deviations consists of three parts: It provides a mathematical justification for a new, lower level LDP and it supplies a method for computing the new, corresponding LDRF in terms of a constrained optimization of the entropy function at hand.  The very contraction principle also predicts a posterior probability conditioned on the constrain; this is the central idea of Gibbs conditioning \cite{DemboZeitouniLDPBook}.  The last aspect is the statistical foundation of MEP. Alternatively stated, maximizing the entropy function under an asymptotic constraint defines a {\em conditional law} on a set of infinitely large samples, independent or correlated.  The Gibbs conditioning problem associated with Sanov's LDP for i.i.d. sample frequencies was carefully studied by van Campenhout and Cover \cite{vanCampenhout_1981}.  This version of the MEP has been applied to a wide range of problems in recent years \cite{dill_rmp}.

\subsection{A conjecture for continuous-time Markov Chains}

Let $\{Y_t\}_{t\geq 0}$ be a continuous-time Markov process with finite state space $\Sigma=\{1,2,...,N\}$ and $Q$-matrix $Q=(q_{ij})_{1\leq i,j\leq N}$ and where $q_{ii}=-\sum\limits_{j\in \{1,2,...,N\}\backslash\{i\}}q_{ij}$. Suppose we observed the process $Y_t$ during time $t\in[0,T] \ , \ T>0$. For each pairs of states $i,j\in \{1,2,...,N\}$ we record the number of jumps from $i$ to $j$ up to time $T$, and we denote it by $\nu^{[0,T]}_{(i,j)}$. Let the integer $N^{[0,T]}_{i}\geq 0$ be the number of visits of the process $\{Y_t\}_{t\geq 0}$ to state $i\in \{1,2,...,N\}$ up to time $T>0$. We record the empirical occupation times
$$m^{[0,T]}_i\equiv \dfrac{1}{N^{[0,T]}_i}\int_0^T \mathbf{1}_{\{Y_t=i\}}dt \ , \ i=1,2,...,N \ .$$
Note that if $N^{[0,T]}_i=0$, then we must have $\displaystyle{\int_0^T \mathbf{1}_{\{Y_t=i\}}dt}=0$ and for any $j\in \{1,2,...,N\}$ we have $\nu^{[0,T]}_{(i,j)}=\nu^{[0,T]}_{(j,i)}=0$. In this case, by convention, we will set $m_i^{[0,T]}=+\infty$. 

Given the frequency event
$$\mathcal{E}_{\{\nu^{[0,T]}_{(i,j)}\}}\equiv \{\nu^{[0,T]}_{(i,j)}, 1\leq i,j\leq N\} \ ,$$
such that
$$
\sum\limits_{j=1}^{N} \nu^{[0,T]}_{(i,j)} - \sum\limits_{j=1}^{N} \nu^{[0,T]}_{(j,i)} 
= \mathbf{1}_{\{i=i_1\}}-\mathbf{1}_{\{i=i_{n+1}\}} \ ,
$$
we consider all length-$n$ strings of chain type
$$X_1 = (i_1,i_2) , X_2 = (i_2,i_3) , ... , X_n = (i_n, i_{n+1}) \ ,$$
satisfying $\mathcal{E}_{\{\nu_{(i,j)}^{[0,T]}\}}$ and such that $n=\sum\limits_{i,j=1}^N \nu^{[0,T]}_{(i,j)}$. 

Given an $(i,j)$ such that $\nu^{[0,T]}_{(i,j)}\geq 1$ on the event $\mathcal{E}_{\{\nu^{[0,T]}_{(i,j)}\}}$, we define by $\text{\#}^{(i,j)}_\ell(\mathcal{E}_{\{\nu_{(i,j)}\}})$ to be the 
number of different strings of chain type $X_1=(i_1,i_2),...,X_n=(i_{n-1}, i_n)$ 
with the $\ell$-th element being $X_\ell=(i,j)$, and satisfying $\mathcal{E}_{\{\nu^{[0,T]}_{(i,j)}\}}$. Let $\mathbf{1}_\ell^{i, \checkmark}$ be the indicator function which is $1$ if there is at least one string of chain type $X_1=(i_1,i_2),...,X_n=(i_{n-1}, i_n)$ satisfying $\mathcal{E}_{\{\nu^{[0,T]}_{(i,j)}\}}$ with $X_\ell=(i, *)$, and it is $0$ otherwise. With all these, we conjecture that

\,

\noindent \textbf{Conjecture.}
Conditioned on $\mathcal{E}_{\{\nu^{[0,T]}_{(i,j)}\}}$ and $m^{[0,T]}_i$ ($i=1,2,...,N$), we have 
$$\lim\limits_{\Delta t\rightarrow 0}\dfrac{1}{\Delta t}
\mathbf{P}\left(\left.Y_{\Delta t}=j\right|\mathcal{E}_{\{\nu^{[0,T]}_{(i,j)}\}}, m_i^{[0,T]}, Y_0=i\right)
=\left\{
\begin{array}{rl}
\mathbf{1}^{i, \checkmark}_1 \cdot q_{ij}^{[0,T]}   \ , & \text{ if } j\neq i \ ;
\\
\\
-\mathbf{1}^{i, \checkmark}_1 \cdot \dfrac{1}{\widehat{m}_i^{[0,T]}} \ , & \text{ if } j=i \ .
\end{array}
\right.$$
Here for $j\neq i$ we have 
$$q_{ij}^{[0,T]}=
\dfrac{1}{\widehat{m}_i^{[0,T]}}
\cdot \dfrac{\mathbf{1}^{j, \checkmark}_2 \cdot \text{\#}^{(i,j)}_1(\mathcal{E}_{\{\nu^T_{(i,j)}\}})}
{\sum_{k_2=1}^N \mathbf{1}^{k_2, \checkmark}_2 \cdot 
\text{\#}^{(i,k_2)}_1(\mathcal{E}_{\{\nu^T_{(i,j)}\}})} \ .$$
As time $T\rightarrow\infty$, we further expect that
$$\lim\limits_{T\rightarrow\infty}\left|\widehat{m}_i^{[0,T]}-m_i^{[0,T]}\right|=0 \ , \ 
\lim\limits_{T\rightarrow\infty} m_i^{[0,T]}=m_i \ , \ 
\lim\limits_{T\rightarrow\infty} q_{ij}^{[0,T]} =q_{ij} \ .$$

\,

We expect to prove the above conjecture using similar arguments as we did in Sections \ref{Sec:MarkovChain}-\ref{Sec:InfiniteTimeLimitMarkovChain} of this paper.

\ 

\noindent{\textbf{Acknowledgements.}} WH would like to thank Louis Waitong Fan, Hao Ge, Yong Liu, Hui He and Yin Ouyang for reading an earlier version of the manuscript and helpful comments, and financial support from the Simons Foundation. HQ acknowledges Chris Burdzy and Jian-Sheng Xie for discussions and the Olga Jung Wan Endowed Professorship.

\bibliographystyle{plain}
\bibliography{bibliography_ProcessConditioned}

\end{document}